\documentclass[11pt]{article}
\usepackage{subcaption}
\usepackage{amsmath,amsthm,amsfonts,amssymb,amscd}
\usepackage{bbm}
\numberwithin{equation}{section}
\usepackage{xcolor}
\usepackage{cite}
\usepackage[scr]{rsfso}
\usepackage{stackrel}
 \usepackage{amsmath}
\usepackage{enumerate} 
\usepackage{cases}
\usepackage{color}
\usepackage{float}
\usepackage{cases}
\usepackage{soul}
\usepackage{graphicx}
\usepackage{mathpazo}
\usepackage{fancyhdr}
\usepackage{empheq}
\usepackage[margin=1in]{geometry}
\def\disp{\displaystyle}

\def\XX{\mathbb{X}}
\def\YY{\mathbb{Y}}

\def\hat{\widehat}

\def\Bar{\overline}
\def\ra{\rangle}
\def\la{\langle}
\def\ve{\varepsilon}
\def\B{\mathbb{B}}
\def\h{\hfill\Box}
\def\R{\mathbb{R}}
\def\ox{\bar{x}}

\def\oy{\bar{y}}

\def\ov{\bar{v}}

\def\cone{\mbox{\rm cone}\,}
\def\ri{\mbox{\rm ri}\,}

\def\Im{\mbox{\rm Im}\,}

\def\gph{\mbox{\rm gph}\,}

\def\dim{\mbox{\rm dim}\,}

\def\dom{\mbox{\rm dom}\,}
\def\Ker{\mbox{\rm Ker}\,}

\def\bd{\mbox{\rm bd}\,}

\def\cl*co{\mbox{\rm cl}^*\mbox{\rm co}\,}

\def\cl{\mbox{\rm cl}\,}

\def\h{\hfill\triangle}
\def\dn{\downarrow}
\def\O{\Omega}

\def\ph{\varphi}

\def\oR{\Bar{\R}}

\def\lm{\lambda}

\def\Lm{\Lambda}

\def\hs7{\hspace*{7pt}}

\def\Id{\mathbb{I}}

\def\lange{\langle}

\renewcommand{\theequation}{\thesection.\arabic{equation}}

\def\h{\hfill\Box}
\def\kk{\kappa}
\begin{document}

\newtheorem{Theorem}{Theorem}[section]
\newtheorem{Conjecture}[Theorem]{Conjecture}
\newtheorem{Proposition}[Theorem]{Proposition}
\newtheorem{Remark}[Theorem]{Remark}
\newtheorem{Lemma}[Theorem]{Lemma}
\newtheorem{Corollary}[Theorem]{Corollary}
\newtheorem{Definition}[Theorem]{Definition}
\newtheorem{Example}[Theorem]{Example}
\newtheorem{Fact}[Theorem]{Fact}
\newtheorem*{pf}{Proof}
\renewcommand{\theequation}{\thesection.\arabic{equation}}
\newcommand*\samethanks[1][\value{footnote}]{\footnotemark[#1]}
\normalsize
\normalfont
\medskip
\def\endproof{$\h$\vspace*{0.1in}}
\title{\bf Stable Recovery of Regularized Linear Inverse Problems}
\date{}
\author{Tran T. A. Nghia\thanks{Department of Mathematics and Statistics, Oakland University, Rochester, MI 48309, USA. Emails: nttran@oakland.edu; hnpham@oakland.edu;  nghiavo@oakland.edu} \and Huy N. Pham \samethanks \and  Nghia V. Vo  \samethanks }

\maketitle
\begin{abstract}
Recovering a low-complexity signal from its noisy observations by regularization methods is a cornerstone of inverse problems and compressed sensing. Stable recovery ensures that the original signal can be approximated linearly by optimal solutions of the corresponding Morozov or Tikhonov regularized optimization problems.  In this paper, we propose new characterizations for stable recovery in finite-dimensional spaces, uncovering the role of nonsmooth second-order information. These insights enable a deeper understanding of stable recovery and their practical implications. As a consequence, we apply our theory to derive new sufficient conditions for stable recovery of the analysis group sparsity problems, including the group sparsity and isotropic total variation problems. Numerical experiments on these two problems give favorable results about  using our conditions to test stable recovery.

\end{abstract}

\vspace{0.1in}


\section{Introduction}
Recovering a signal $x_0$ in a Euclidean space $\XX$ from its observation $y_0=\Phi x_0$ in another Euclidean space $\YY$ over a linear operator $\Phi \in \mathscr{L}(\XX,\YY)$ is a typical linear inverse problem in science and engineering. Solving the linear  system 
\begin{equation}\label{p:LE}
    \Phi x=y_0
\end{equation}
cannot recover exactly $x_0$ in general; especially when the system is ill-posed, there may be infinitely many solutions. By leveraging the prior low-complexity information of $x_0$ such as sparsity, group sparsity, or low-rank, adding a regularizer $R:\XX\to \R$ as an objective function and solving the following optimization problem
\begin{equation}\label{p:BP}
 \min_{x\in \XX}   \quad R(x)\quad \mbox{subject to}\quad \Phi x=y_0 
\end{equation}
is a very successful technique in recovering $x_0$ \cite{CT05,DET05,CR09,FR13,GSH11,BB18,VPF15}. A classical choice for $R(x)$ is the Euclidean norm squared \cite{EHN96,M93,TA77}. In this case,  problem \eqref{p:BP} gives the minimum norm solution to the linear system, which is the projection from the origin $0\in \XX$ to the feasible affine space \eqref{p:LE}, also known as  the {\em Moore-Penrose operator} or the {\em generalized inverse} of $\Phi$ at $y_0$. However, this solution is usually not the original signal $x_0$ that we want to recover. In many applications in compressed sensing, image processing, and machine learning, the regularizer $R$ is chosen as a {\em nonsmooth} convex function such as the $\ell_1$ norm, the $\ell_1/\ell_2$ norm, the total variation seminorm, and the nuclear norm; see, e.g.,   \cite{BDE09,BB18,CR09,CP11,FR13} for the influence of  this regularization technique in different areas.
 
Figure 1 below illustrates the success rate of the above regularization method when $\Phi\in \R^{m\times n}$ is a uniform Gaussian matrix for  problem \eqref{p:BP} in two different cases: (a) The {\em group sparsity regularization problems} \cite{YL06,G11,FPVDS13,RRN12} with the $\ell_1/\ell_2$ norm  $R(x)=\|x\|_{1,2}$ in $\R^n$ of recovering signals $x_0\in\R^{2000}$ with $100$ non-overlap groups of  $20$ elements and $10$ nonzero active groups; (b) The {\em isotropic total variation problems} \cite{ROF92,CP11} with the regularizer  $R(x)=\|\nabla x\|_{1,2}$ of recovering images $x_0\in \R^{28\times 28}$, where $\nabla x$ is the {\em discrete gradient operator} of image $x\in \R^{28\times 28}$. For each $m$, we use the {\em cvx package} \cite{DB16} to solve $100$ problems \eqref{p:BP}   generated with  $100$ different random matrices $\Phi$. The green curve measures the percentage of problems that are able to recover the original signals/images; see our Section~5 for further details. 

\begin{figure}[h]
\centering
\centering
\includegraphics[width=1\linewidth]{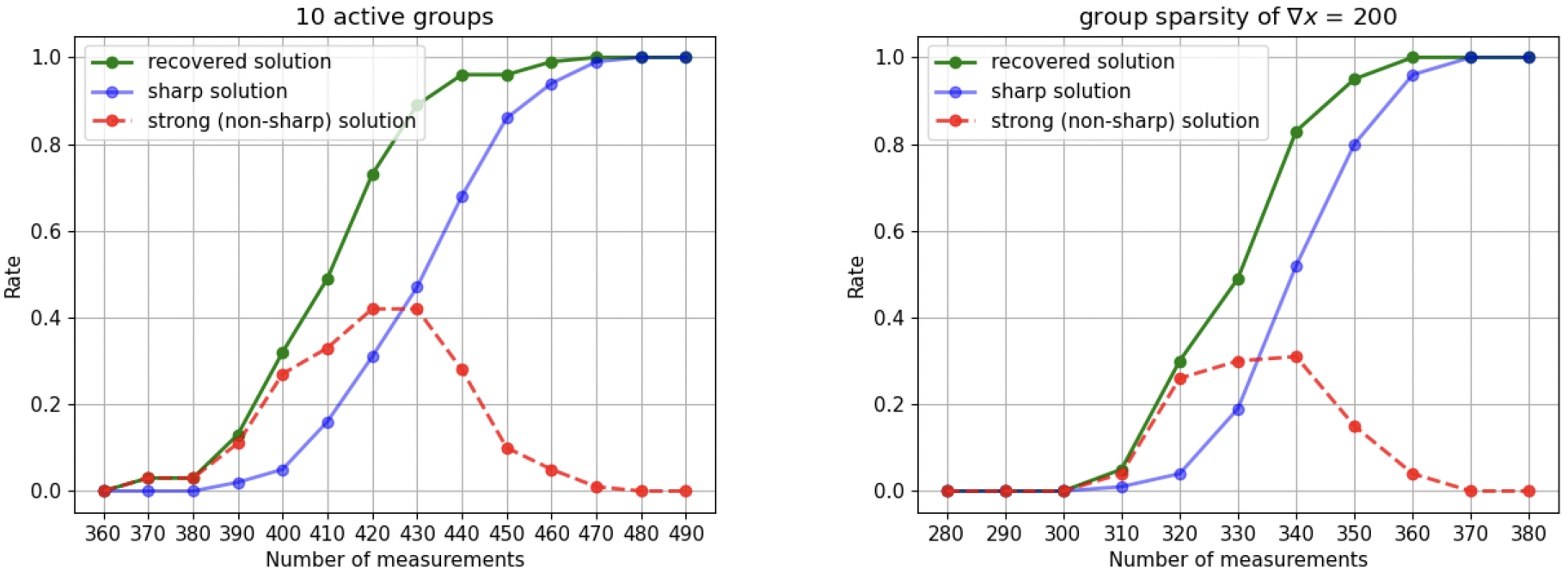} 

\caption{Proportion of cases for which $x_0$ can be recovered as a function of $m$.}
\end{figure}

 When there is noise or small perturbation presented in the observation $y\in \YY$ with $\|y-\Phi x_0\|\le \delta$ for some error level $\delta>0$, the original $x_0$ may be recovered approximately by solving some related optimization problems to \eqref{p:BP} such as the {\em Morozov regularization problem} \cite{M93}
\begin{equation}\label{p:BP1}
 \min_{x\in \XX}   \quad R(x)\quad \mbox{subject to}\quad \|\Phi x-y\|\le \delta,
\end{equation} 
or its {\em Lagrangian relaxation} (also known as the {\em Tikhonov regularization problem} \cite{TA77})
\begin{equation}\label{p:Lass}
 \min_{x\in \XX}\quad  \frac{1}{2}\|\Phi x-y\|^2+\mu R(x) 
\end{equation} 
with $\mu>0$ being known as the {\em tuning or regularization parameter}. When $R(x)=\|x\|_1$ is the $\ell_1$ norm that is well-known to  promote sparsity for optimal solutions and $\mu$ is proportional to $\delta$, \cite{CT05,GSH11} showed that any solution $x_\delta$ of \eqref{p:BP1} and $x_\mu$ of \eqref{p:Lass} satisfy
\begin{equation}\label{eq:Com}
    \|x_\delta-x_0\|=O(\delta)\quad \mbox{and}\quad \|x_\mu-x_0\|=O(\delta)
\end{equation}
if and only if $x_0$ is a unique optimal solution of problem \eqref{p:BP}. The property \eqref{eq:Com} will be referred to as {\em stable recovery} \cite{CRT06,H13,FPVDS13} at $x_0$ in our paper (also known as the  linear convergence rate for regularization methods in inverse problems literature, e.g., \cite{GSH11,G11}). This elegant characterization for stable recovery can be extended to the class of  {\em piecewise linear convex regularizers} due to their spontaneous {\em polyhedral} structures. Beyond this class,  various sufficient conditions have been proposed to ensure stable recovery; for example, the {\em Null Space Property} established in \cite{DX01} for $\ell_1$ problems and extended for more general regularizers in \cite{FPVDS13,FR13,VGFP15,VPF15}, the {\em Nondegeneracy Source Condition} together with the {\em Restricted Injectivity} \cite{GSH11, CR13, VPF15}, and the {\em Minimum Gain Property} of the linear operator $\Phi$ over the {\em descent cone} at $x_0$ \cite{CRPW12}. In the recent paper \cite{FNT23}, these properties at $x_0$ are proved to be equivalent to {\em Sharp Minimum Property} at $x_0$ for problem \eqref{p:BP}, which was introduced by Polyak in \cite{P87} and Crome \cite{C78} with a different name {\em Strong Uniqueness}. In Figure~1,  the blue curve demonstrates the percentage of problems that recover sharp optimal solutions based on a numerical characterization for sharp minima in \cite{FNT23} with respect to the number of measurements $m$. Not all recovered solutions of the $\ell_1/\ell_2$ problems in Figure~1 are sharp, but they are all {\em strong optimal solutions} \cite{FNT23}. The gap between the green curve and blue curve is illustrated by the red one, which signifies the percentage of cases of strong (non-sharp) optimal solutions of problem \eqref{p:BP}; see our Section~5 for further explanations and numerical experiments. 

Motivated by \cite{CT05,GSH11,G11,H13,FNT23, FPVDS13,VPF15}, it is natural to question: Does stable recovery always occur at $x_0$ when it is an  unique optimal solution of problem \eqref{p:BP}? The quick answer is no; see, e.g., our Example~\ref{ex:NoSR} for a simple group sparsity regularization problem. Unlike \cite{CT05,GSH11}, solution uniqueness is not enough to guarantee stable recovery in {\em non-polyhedral} regularized linear inverse problems. As most of papers in this direction have used sufficient conditions for stable recovery that lead to sharp minima, another intuitive  question should be: Can stable recovery occur at the unique non-sharp  optimal solution $x_0$ of problem~\eqref{p:BP}? Our simple Example~\ref{ex:l12} gives an affirmative answer for a group sparsity regularization problem, at which $x_0$ is a unique non-sharp optimal solution of ~\eqref{p:BP}, i.e., the Null Space Property \cite{DX01,FPVDS13,FR13,VGFP15,VPF15}, the Nondegeneracy Source Condition and Restricted Injectivity \cite{GSH11,CR13,VPF15}, and the Minimum Gain Property \cite{CRPW12} all fail. Stable recovery does occur beyond sharp minima and those conditions. It brings back the fundamental  question that we will try to answer throughout the paper: When does stable recovery occur at (unique non-sharp) optimal solutions of nonpolyhedral regularized  optimization problems \eqref{p:BP}?\\


\noindent {\bf Our contributions.} 
\begin{itemize}
\item We derive a full characterization of stable recovery for the general case of convex regularizers in Theorem~\ref{thm:Rob}. The theorem states that the optimal solution  $x_0$ of problem~\eqref{p:BP} is {\em stably recoverable} in the sense of \eqref{eq:Com} if and only if 
\begin{equation}\label{eq:CSR}
    \Ker \Phi\cap T_{\partial R^*({\rm Im}\, \Phi^*)}(x_0)=\{0\}, 
\end{equation}
where $ \Ker \Phi$ is the kernel/null space of $\Phi$ and  $T_{\partial R^*({\rm Im}\, \Phi^*)}(x_0)$ is the {\em contingent/tangent cone} \cite{AF90,RW98} acting to the image of the {\em subdifferential mapping} of the {\em Fenchel conjugate} regularizer $R^*$ over the range/image of the adjoint operator of $\Phi$ (i.e.  $\Phi^*$) at $x_0$. This is a second-order condition, as the action of the tangent cone is considered taking another ``derivative'' from the subdifferential set $\partial R^*(\Im \Phi^*)$. To the best of our knowledge, all the previous papers in this direction used first-order analysis to obtain sufficient conditions for stable recovery. Our paper reveals  second-order information for stable recovery. This truly distinguishes our approach from other aforementioned papers. Some conditions similar to \eqref{eq:CSR} were obtained recently in \cite{FNP23,FNP24} to characterize strong minima of problem~\eqref{p:BP}, but strong solutions are not enough to guarantee \eqref{eq:CSR} or stable recovery. Sharp minimum property at $x_0$ implies \eqref{eq:CSR}, but not vice versa. When $R$ is a twice differentiable function, condition \eqref{eq:CSR} holds provided that  $x_0$ is a strong solution. As strong minima often incorporate second-order properties \cite{BS00,RW98}, this reconfirms the hidden second-order information behind stable recovery. We also extend \cite[Theorem~4.7]{GSH11} to a broad class of {\em convex piecewise linear-quadratic functions} \cite{RW98} by proving  that stable recovery occurs at $x_0$ if and only if it is a unique optimal solution of problem~\eqref{p:BP}. This class contains many important regularizers in statistics, optimization, and machine learning such as the {\em elastic net} \cite{ZH05}, {\em Huber norm} \cite{H73}, and the {\em discrete Blake-Zisserman regularizer} \cite{BZ87}. It is worth noting that a unique optimal solution of the convex piecewise linear-quadratic function is actually a strong one. This again gives the feeling that stable recovery is closer to the strong minimum property than the sharp minimum one. 

\item We apply our theory to the particular class of analysis group sparsity seminorm $R(x)=\|D^*x\|_{1,2}$, where $D$ is an ${n\times p}$ matrix and $\|\cdot\|_{1,2}$ is the $\ell_1/\ell_2$ norm in $\R^p$ that has been popularly used in statistics \cite{YL06} and image processing \cite{CP11,ROF92}. When $x_0$ is a unique solution of problem \eqref{p:BP}, \cite[Corollary~5.10]{FNT23} showed that both $x_\delta$ and $x_\mu$ converge to $x_0$ with rate $\delta^\frac{1}{2}$. We advance that result by proving that the convergence rate can be improved to the linear rate $\delta$, i.e., stable recovery occurs under some new sufficient conditions in Corollary~\ref{Cor:Equi} and Corollary~\ref{cor:Id} that are quite simple to verify and  totally independent of sharp minima  \cite{FPVDS13,FNT23,G11,H13} and other aforementioned conditions \cite{FR13,CR13,CRPW12}.

\item Finally, we provide numerical experiments for the case of group sparsity and isotropic total variation regularization problems. Our experiments indicate some interesting phenomena that most of strong (non-sharp) optimal solutions of  group sparsity regularization problems are stably recoverable. 


\end{itemize}

{\bf Outline of the Paper.} Section 2 provides preliminary on basic variational analysis used throughout the paper. Section 3 presents our main result Theorem~\ref{thm:Rob}, a second-order characterization of stable recovery for general regularized linear inverse problems  together with its consequences when $R$ is a smooth convex function or a convex piecewise linear-quadratic function. In Section 4, we apply our results to the analysis group sparsity problems and  establish new sufficient conditions for stable recovery that are independent of sharp minima. Section 5 provides numerical experiments for group sparsity and isotropic total variation problems, illustrating the practical relevance of our theory and motivating future research directions in Section 6.

\section{Preliminary}
Throughout the paper, $\XX$ and $\YY$ are  Euclidean spaces. We denote  $\la\cdot,\cdot\ra$ by the inner product in these spaces  and $\|\cdot\|$ by the corresponding  Euclidean norm.  The set $\B_r(x)$ represents the closed ball  in $\XX$ with center $x\in \XX$ and radius $r>0$.  For any linear operator $\Phi\in \mathscr{L}(\XX,\YY)$ between the two Euclidean spaces, we denote  $\Ker \Phi$ and $\Im \Phi$ by the kernel/null space of $\Phi$ and the range/image space of $\Phi$, respectively.

Let $\ph:\XX \to \oR:=\R\cup\{\infty\}$ be a proper lower semi-continuous (l.s.c.) extended real-valued convex function with the domain $\dom \ph:=\{x\in \XX|\; \ph(x)<\infty\}$. The {\em Fenchel conjugate} of $\ph$ is also a l.s.c. extended real-valued convex function $\ph^*:\XX\to \oR$ defined by 
\begin{equation}\label{eq:Fen}
    \ph^*(v):=\sup\{\la v,x\ra-\ph(x)|\, x\in \XX\}\quad \mbox{for}\quad v\in \XX.  
\end{equation}
The {\em subdifferential} of $\ph$ at $x\in\dom \ph$ is described by
 \[
 \partial \ph(x):=\{v\in \XX|\; \ph(u)-\ph(x)\ge \la v, u-x\ra\;\mbox{for all}\; u\in \XX\}. 
 \]
It follows that 
\begin{equation}\label{eq:FR}
     \partial \ph(x)=\{v\in \XX|\, \ph^*(v)+\ph(x)=\la v,x\ra\}.
\end{equation}
A particular example of convex function that can take infinity value is the {\em indicator function} to a closed convex set $C\subset \XX$ denoted by $\iota_C(x)$, that is  $0$  if $x\in C$ and $\infty$ otherwise. The subdifferential of $\delta_C$ at $x\in C$ is known as the {\em normal cone} to $C$ at $x$:
\begin{equation}\label{eq:Nor}
    N_C(x):=\{v\in \XX|\; \la v, u-x\ra\le 0\; \mbox{for all}\; u\in C\}.
\end{equation}
Throughout the paper, $\ri C$ represents the set of all {\em relative interior} points of $C$
\begin{equation}\label{def:Ri}
\ri C:=\{x\in C|\; \exists\, \epsilon>0:\; \mathbb{B}_\epsilon(x)\cap{\rm aff}\,C\subset C\},
\end{equation}
where ${\rm aff}\,C$ is the {\em affine hull} of $C$. Moreover, $\cone C$ is denoted by the {\em conic hull} of $C$; see \cite{R70} for these important notions in convex analysis.

The following geometric definition of the {\em contingent/tangent cone} \cite{AF90,BS00,RW98} plays a crucial role in our analysis.  
\begin{Definition}[Tangent/contingent cone] Let $\Omega$ be a closed (possibly nonconvex) set of $\XX$. The contingent/tangent cone to $\Omega$ at $x_0 \in \O$ is defined by
\begin{equation}\label{eq:Tangent}
T_{\O}(x_0) := \{w \in \XX |\, \exists\, t_k \downarrow 0, w_k \to w: x_0 + t_kw_k \in \Omega \}.
\end{equation}
If additionally, $\O$ is convex, we have
\begin{equation}\label{eq:TanC}
T_{\O}(x_0)= \cl\left[\cone(\O-x_0)\right],
\end{equation}
which is the closure of the conic hull of $\O-x_0$. 
\end{Definition}

Next, let us recall the definitions of sharp and strong optimal solutions/minimizers  \cite{BS00,C78,P87,RW98}, the two key players  in this paper. 
 
\begin{Definition}[Sharp and strong minima] We say that the function $\ph:\XX\to \oR$ has a sharp minimum at  $\ox\in \dom \ph$ or $\ox$ is a sharp optimal solution/minimizer of $\ph$ if there exist constants $c,\ve>0$ such that 
\begin{equation}\label{eq:Sha}
    \ph(x)-\ph(\ox)\ge c\|x-\ox\|\quad \mbox{for all}\quad x\in  \B_\ve(\ox).
\end{equation}
Moreover, the function $\ph$ is said to have a strong minimum at $\ox$  or $\ox$ is a strong optimal solution/minimizer of $\ph$ if  there exist constants $\kk, \delta>0$ such that 
\begin{equation}\label{eq:Str}
    \ph(x)-\ph(\ox)\ge \frac{\kk}{2}\|x-\ox\|^2\quad \mbox{for all}\quad x\in  \B_\delta(\ox).
\end{equation}
\end{Definition}
When $\ph$ is a convex function, it is well-known that $\ve$ in \eqref{eq:Sha} can be chosen as infinity, i.e., sharp minimum at $\ox$ is a global property for convex functions. On the other hand, strong minimum at $\ox$ is usually a local property. Sharp minimum introduced by Polyak \cite{P87} and Crome \cite{C78} with a different name ``Strong uniqueness'' is an important notion in optimization that has substantial impacts on convergence analysis in many iterative algorithms.  When the function $\ph$ is a  {\em piecewise linear convex} function \cite[Definition~2.47]{RW98}, i.e., its epigraph is polyhedral,  $\ox$ is a sharp optimal solution of $\ph$ if and only if it is  a unique optimal solution. This is one of the main reasons that solution uniqueness is enough to characterize stable recovery  for $\ell_1$ linear inverse problems \cite{FNT23,GSH11}. 

Obviously, the property of strong minimum is stronger than sharp minimum. Beyond polyhedral optimization structures, strong minimum is a more conventional property that has been used extensively in nonlinear (possibly nonconvex) optimization problems \cite{BS00,RW98} with different applications to stability theory and algorithms. To characterize sharp and strong optimal solutions, it is standard to use {\em directional derivative} and {\em second  subderivative}  \cite{P87,BS00,RW98} of convex functions.
\begin{Definition}[directional derivative and second subderivative]  The directional derivative of convex function $\ph$ at $\ox\in \dom \ph$ is the function $d\ph(\ox):\XX \to\oR$  defined by
\begin{equation}\label{eq:DD}
d\ph(\ox)(w):=\lim_{t\dn 0}\dfrac{\ph(\ox+tw)-\ph(\ox)}{t}\quad \mbox{for}\quad w\in \XX.
\end{equation}
The second subderivative of $\ph$ at $\ox$ for $\ov\in \partial \ph(\ox)$ is the function $d^2 \ph(\ox|\ov):\XX\to \oR$  defined by   
\begin{equation}\label{eq:SD}
d^2\ph(\ox|\ov)(w): =\liminf_{t\dn 0, w^\prime \to w}\dfrac{\ph(\ox+tw^\prime)-\ph(\ox)-t\la \ov, w^\prime\ra}{\frac{1}{2}t^2}\quad \mbox{for} \quad w\in \XX.
\end{equation}
\end{Definition}
When $\ph$ is continuous around $\ox$, it is well-known that the subdifferential set $\partial \ph(\ox)$ is  nonempty and compact  in  $\XX$. Note from  \cite[Theorem~8.30]{R70} that 
\begin{equation}\label{eq:DDmax}
    d\ph(\ox)(w)=\max\{\la v,w\ra|\; v\in \partial \ph(\ox)\}\quad \mbox{for}\quad w\in \XX.
\end{equation}
The calculation of $d^2\ph(\ox|\ov)(w)$ is more complicated; see  \cite{BS00,RW98} for its computation in different classes of functions. 
When $\ph$ is twice-differentiable at $\ox$, we have
\begin{equation}\label{eq:smooth}
d^2\ph(\ox|\, \nabla \ph(\ox))(w)=\la  w,\nabla^2\ph(\ox)w\ra\quad \mbox{for}\quad w\in \XX.
\end{equation}
The following simple sum rule for second subderivative \cite[Exercise~13.18]{RW98} is helpful. Let  $\ph, \phi:\XX\to\oR$ be proper l.s.c. convex functions with $\ox\in {\rm int}\,(\dom \ph)\cap\dom \phi$. Suppose that $\ph$ is twice differentiable at $\ox$ and $\ov\in \partial (\ph+\phi)(\ox)$. Then we have
\begin{equation}\label{eq:sum}
d^2(\ph+\phi)(\ox|\,\ov)(w)=\la w,\nabla^2\ph(\ox)w\ra+d^2\phi(\ox|\, \ov-\nabla \ph(\ox))(w)\quad \mbox{for}\quad w\in \XX.
\end{equation}
The next result is taken from \cite[Lemma~3, page 136]{P87} and \cite[Theorem~13.24]{RW98} to characterize sharp and strong optimal solutions.

\begin{Lemma}[Characterizations for sharp and strong solutions]\label{Fa} Let $\ph:\XX \to \oR$ be a proper l.s.c convex function with $\ox\in \dom \ph$. We have: 

\begin{enumerate}[{\rm (i)}]
\item $\ox$ is a sharp minimizer of $\ph$ if and only if there exists $c>0$ such that $d\ph(\ox)(w)\ge c\|w\|$ for all $w\in \XX$, i.e., $d\ph(\ox)(w)>0$ for all $w\in \XX \setminus\{0\}$.

\item $\ox$ is a strong  minimizer of $\ph$ iff $0\in \partial \ph(\ox)$ and there exists $\kk>0$ such that $d^2\ph(\ox|\, 0)(w)\ge \kk\|w\|^2$ for all $w\in \XX$, i.e., $d^2\ph(\ox|\, 0)(w)>0$ for any $w\neq 0$.
\end{enumerate}
\end{Lemma}

\section{Full characterizations for stable recovery of regularized linear inverse problems}
\setcounter{equation}{0}
In this section, we mainly study the following regularized linear inverse problem
\begin{equation}\label{p:P0}
    \min_{x\in \XX}\quad R(x)\quad \mbox{subject to}\quad \Phi x=y_0,
\end{equation}
where $\Phi\in \mathscr{L}(\XX,\YY)$ is a linear operator between two Euclidean spaces $\XX$ and $\YY$, $R:\XX\to \R$ is a continuous convex regularizer, and $y_0$ is a known vector in $\YY$. A feasible solution $x_0$ of problem~\eqref{p:P0} is a minimizer if and only if $0\in \partial R(x_0)+\Im \Phi^*$, i.e., there exists a {\em dual certificate}  $v\in \partial R(x_0)\cap \Im \Phi^*$. The set of all dual certificates is defined by
\begin{equation}\label{eq:DC}
\Delta(x_0):= \partial R(x_0)\cap \Im \Phi^*. 
\end{equation}
 As discussed in the Introduction section, when there is noise in the observation, it is conventional to solve the following regularized least-square optimization problem
\begin{equation}\label{p:Las}
P(y,\mu)\qquad    \min_{x\in \XX}\quad \frac{1}{2}\|\Phi x-y\|^2+\mu R(x),
\end{equation}
with parameter $y\in \YY$ satisfying $\|y-y_0\|\le \delta$ for some noise level $\delta>0$ and regularization parameter $\mu=c\delta$ for some constant $c>0$. The primary objective of our paper is the property of stable recovery \cite{CT05,CRT06,GSH11,H13,FPVDS13}, at which the optimal solution $x_0$ of \eqref{p:P0} can be recovered approximately by solving problem \eqref{p:Las} and the convergence rate has the same order with $\delta$.  

\begin{Definition}[Stable recovery] Let $x_0$ be an optimal solution of problem~\eqref{p:P0}. 
We say that  the {\em stable recovery} occurs at $x_0$ or $x_0$ is {\em stably recoverable} if  there exists a positive constant $K$ such that for sufficiently small $\delta>0$, any optimal solution $x(y,\mu)$ of problem \eqref{p:Las}   satisfies
\begin{equation}\label{eq:Est}
\|x(y,\mu)-x_0\|\le K\delta, 
\end{equation} 
whenever $\|y-y_0\|\le \delta$ and $\mu$ is chosen proportionally to $\delta$, i.e., $ \mu=c\delta>0$ for any positive constant $c$.
\end{Definition}
This definition is consistent with \eqref{eq:Com} in the Introduction section. Although the property of stable recovery in \eqref{eq:Com} also relates to the Morozov regularization problem \eqref{p:BP1}, we focus our study on the Tikhonov regularization problem \eqref{p:Las}, as the analysis for \eqref{p:BP1} is similar and solving the unconstrained problem \eqref{p:Las} seems to be easier. It is worth noting that stable recovery studied in other papers \cite{CT05,FPVDS13,H13} is usually described in a global setting, i.e., estimate \eqref{eq:Est} holds for any $\delta>0$. Our definition of stable recovery has the local sense, as $\delta$ is small. This local perspective aligns well with practical scenarios when the noise or perturbation level is typically minor. In inverse problems  literature, this property is also known as the {\em linear convergence} of $x(y,\mu)$ to $x_0$; see, e.g.,  \cite{GSH11,G11}. 

When the regularizer $R(x)$ is the $\ell_1$ norm in $\R^n$, stable recovery occurs at $x_0$ if and only if $x_0$ is a unique minimizer of \eqref{p:P0} \cite{CT05,GSH11}. For nonpolyhedral regularizers, solution uniqueness is not enough to guarantee stable recovery. Many sufficient conditions are proposed for stable recovery in more general settings such as the Null Space Property \cite{DX01,FR13,FPVDS13,VGFP15,VPF15}, the Nondegeneracy Source Condition and Restricted Injectivity \cite{GSH11,CR13,VPF15}, and the Minimum Gain Property \cite{CRPW12}. In the recent paper \cite{FNT23}, it is shown that all of these properties lead to the sharp minimum at $x_0$ for problem~\eqref{p:P0} in the sense that $x_0$ a sharp optimal solution of the function 
\begin{equation}\label{eq:ph}
\ph(x):=R(x)+ \iota_{\Phi^{-1}(y_0)}(x).
\end{equation}
However, the following example demonstrates that stable recovery may occur without sharp minima or any aforementioned properties.

\begin{Example}[Stable recovery without sharp minima]\label{ex:l12}
\rm Consider the following $\ell_1/\ell_2$  problem
\begin{equation}\label{p:E123}
\min_{x\in \R^3}\quad  R(x) = \sqrt{x_1^2+x_2^2}+ |x_3| \quad \mbox{subject to}\quad \Phi x=y_0   
\end{equation}
with $\Phi=\begin{pmatrix} 1&1&0\\1&1&1\end{pmatrix}
$,  $x_0=(1,1,0)^T$, and $y_0=(2,2)^T$. Note that $\Ker \Phi = \R(1,-1,0)^{T}$. Moreover, we have $\partial R(x_0) = \left(\frac{1}{\sqrt{2}},\frac{1}{\sqrt{2}}\right) \times [-1,1]
$ and $\Delta(x_0)=\left(\frac{1}{\sqrt{2}},\frac{1}{\sqrt{2}}\right) \times [-1,1]$, which implies that  $x_0$ is an optimal solution. Additionally, for all feasible solutions $x=x_0+t(1,-1,0)^T$, we have 
\[
\ph(x)-\ph(x_0)= \sqrt{(1+t)^2+(1-t)^2} - \sqrt{2
} = \dfrac{2t^2}{\sqrt{(1+t)^2+(1-t)^2}+\sqrt{2}}\in \left[\frac{\|x-x_0\|^2}{3},\frac{\|x-x_0\|^2}{2\sqrt{2}}\right]
\]
when $x$ is close to $x_0$. This tells us that $x_0$ is the unique  and strong optimal solution of function $\ph$ or problem \eqref{p:E123}. Obviously, $x_0$ is not a sharp minimizer by the definition in \eqref{eq:Sha}. 

Let us show next that stable recovery occurs at $x_0$. For any $y$ satisfying $\|y-y_0\| \leq \delta$, $x$ is an optimal solution of $P(y, \mu)$ with $\mu = c\delta$ if and only if  
$-\Phi^*(\Phi x - y)  \in \mu \partial R(x).$ 
This can be expressed by the following system
\begin{equation*}
\begin{cases}
(y_1 + y_2-(2x_1+2x_2+x_3), y_1 + y_2-(2x_1+2x_2+x_3)) \in \mu \partial\|\cdot\|(x_1,x_2) \\
y_2-(x_1+x_2+x_3) \in \mu \partial |x_3|.
\end{cases}
\end{equation*}
When $\delta$ is sufficiently small, $(y_1,y_2)$ is close to $(2,2)$. If $(x_1,x_2) = (0,0)$, we have  $|y_1 + y_2 - x_3| \le \frac{\mu}{\sqrt{2}}$ and $|y_2 -x_3| \le \mu$, which gives $|y_1| \le \mu+\frac{\mu}{\sqrt{2}}$ (contradiction). Thus, we have  $(x_1,x_2) \neq (0,0)$. The above system  turns into
\begin{equation}
\begin{cases}\label{eq:fer1}
 \left(y_1+y_2-(2x_1+2x_2+x_3),y_1+y_2-(2x_1+2x_2+x_3)\right) = \left(\dfrac{\mu x_1}{\sqrt{x^2_{1} + x^2_{2}}},\dfrac{\mu x_2}{\sqrt{x^2_{1} + x^2_{2}}}\right) \\
y_2-(x_1+x_2+x_3) \in \mu \partial |x_3|,
\end{cases}
\end{equation}
 which clearly gives us $x_1=x_2$. 

{\bf Case I:} $x_3 \neq 0$.  We have  $
 y_1+y_2-(4x_1+x_3) = \dfrac{\mu}{\sqrt{2}}$ and $
y_2-(2x_1+x_3) = \pm\mu.$
By solving for $x_1, x_3$, we have  $x_1 = x_2 = \frac{1}{2}(y_1\pm\mu-\frac{\mu}{\sqrt{2}})>0$ and $x_3=y_2-y_1\mp2\mu+\frac{\mu}{\sqrt{2}}$. As $|y_1-2|,|y_2-2|\le \delta$ and $\mu=c\delta$, some simple algebras give us  $
\|x-x_0\| \le C\delta
$ with constant $C>0$. 

{\bf Case II:} $x_3=0$. It follows from \eqref{eq:fer1} that  $
x_1=x_2 = \frac{1}{4}\left( y_1 + y_2- \frac{\mu}{\sqrt{2}}\right)>0$, which also implies $\|x-x_0\|\le C\delta$.
 
From the above two cases,  $x_0$ is stably recoverable by the definition. But $x_0$ is not a sharp solution, i.e., all the conditions mentioned before this example are not satisfied either. $\hfill{\triangle}$
\end{Example}

The next central theorem of our paper gives a full characterization of stable recovery.

\begin{Theorem}[Geometric characterization of stable recovery]\label{thm:Rob} 
Suppose that $R$ is a continuous convex  function and that  $x_0$ is an optimal solution of  \eqref{p:P0}. 
 Then $x_0$ is stably recoverable  if and only if 
\begin{equation}\label{con:SOSC}
\Ker \Phi\cap T_{\partial R^*({\rm Im}\, \Phi^*)}(x_0)=\{0\}.
\end{equation}

\end{Theorem}
\begin{proof} First, let us verify \eqref{con:SOSC} when  stable recovery occurs at $x_0$. Pick any $w\in \Ker \Phi \cap T_{\partial R^*({\rm Im}\, \Phi^*)}(x_0)$. From the definition of the tangent cone \eqref{eq:Tangent}, we find sequences $t_k\dn 0$ and $w_k\to w$ such that $x_k:=x_0+t_kw_k\in \partial R^*({\rm Im}\, \Phi^*)$. Hence, there exists $u_k\in \YY$ such that $x_k\in \partial R^*(\Phi^*u_k)$, which means $\Phi^* u_k\in \partial R(x_k)$. As $R$ is a continuous convex function, it is locally Lipschitz around $x_0$ with some modulus $L>0$. It follows that  $\partial R(x_k)\subset L\B$, where $\B$ is the closed unit ball in $\XX$. The latter tells us that  $\|\Phi^* u_k\|\le L$ for sufficiently large $k$. It is well-known \cite[Proposition~2.173]{BS00} that there exists some constant $s>0$ such that 
\[
d(x;\Ker \Phi^*):=\inf\{\|x-u\||\, u\in \Ker \Phi^*\}\le s\|\Phi^* x\|\quad \mbox{for all}\quad x\in \XX.  
\]
Since $\|\Phi^* u_k\|\le L$, this inequality allows us to find  $\bar u_k\in \YY$ with $\bar u_k-u_k\in \Ker \Phi^*$ and $\|\bar u_k\|\le sL$.
Let us set 
\begin{equation}\label{eq:my}
\mu_k:=\|\Phi(x_k-x_0)\|=t_k\|\Phi w_k\|\to 0
\quad \mbox{and}\quad  y_k:=\Phi x_k+\mu_k \bar u_k\to y_0.
\end{equation}
It follows that
\[
\|y_k-y_0\|=\|\Phi (x_k-x_0)+\mu_k\bar u_k\|\le \|\Phi (x_k-x_0)\|+\mu_k\|\bar u_k\|=\mu_k(1+\|\bar u_k\|)\le \mu_k(1+sL).
\]
Define $\delta_k:=(1+sL)\mu_k$. As $\bar u_k-u_k\in \Ker \Phi^*$, we have 
\[
-\Phi^*(\Phi x_k-y_k)=\mu_k\Phi^*\bar u_k=\mu_k\Phi^* u_k\in \mu_k\partial R(x_k),
\]
which tells us that $x_k$ is an optimal solution of problem  $P(y_k,\mu_k)$ in \eqref{p:Las}. As $\mu_k=(1+sL)^{-1}\delta_k$ and $\delta_k\to 0$, it follows from \eqref{eq:Est} and \eqref{eq:my} that 
\[
t_k\|w_k\|=\|x_k-x_0\|\le K\delta_k=K(1+sL)\mu_k=K(1+sL)t_k\|\Phi w_k\|
\]
for some constant $K>0$. This clearly  implies that $\|w_k\|\le K(1+sL)\|\Phi w_k\|$. As $w_k\to w\in \Ker \Phi$, we get from the latter that $\|w\|\le 0$, i.e., $w=0$. Condition \eqref{con:SOSC} is verified.  

Conversely, we claim that  the stable recovery occurs at $x_0$ when  \eqref{con:SOSC} is satisfied. For any $v\in \Delta(x_0)=\Im \Phi^*\cap \partial R(x_0)$, note that  $\partial R^*(v)\subset \partial R^*(\Im \Phi^*)$. It follows that 
\[
\{0\}=\Ker \Phi\cap T_{\partial R^*({\rm Im}\, \Phi^*)}(x_0)\supset \Ker \Phi\cap{\rm cone}\, (\partial R^*(v)-x_0).
\]
\cite[Theorem~4.5]{FNP24} tells us that  $x_0$ is the unique solution of problem \eqref{p:P0}.

Let us suppose by contradiction that  stable recovery does not occur at $x_0$. Hence, there exist  sequences $\mu_k=c\delta_k\dn 0$ for some constant $c>0$ and   $y_k\to y_0$  such that  $\|y_k-y_0\|\le \delta_k$ and $\|x_k-x_0\|\ge k\delta_k$, where $x_k$ is some optimal solution of $P(y_k,\mu_k)$. Define $t_k:=\|x_k-x_0\|>0$ and $w_k:=t_k^{-1}(x_k-x_0)$. Since $x_0$ is the unique solution of problem~\eqref{p:P0}, it follows from \cite[Proposition~3.1 and Remark~3.2]{FNT23} that $x_k\to x_0$. By passing to subsequences, suppose without loss of generality that $t_k\dn 0$ and $w_k\to w\in \XX$ with  $\|w\|=1$. Since $x_k$ is an optimal solution of problem~\eqref{p:Las},   we have
\begin{equation}\label{eq:In2}
v_k:=-\frac{1}{\mu_k}\Phi^*(\Phi x_k-y_k)=-\frac{1}{c\delta_k}\Phi^*(t_k\Phi w_k-y_k+y_0)\in  \partial R(x_k)=\partial R(x_0+t_kw_k).    
\end{equation}
Hence $\|v_k\|\le L$ for sufficiently large $k$, as $R$ is locally Lipschitz around $x_0$ with  modulus $L>0$. Moreover, the above inclusion tells us that 
\[
x_0+t_kw_k\in \partial R^*(v_k)\subset \partial R^*(\Im \Phi^*), 
\]
which implies that $w\in T_{\partial R^*({\rm Im}\, \Phi^*)}(x_0)$. Note also from \eqref{eq:In2}  that 
\[
\Phi^*(\Phi w_k)=-\frac{c\delta_k}{t_k}v_k+\frac{\Phi^*(y_k-y_0)}{t_k}.
\]
As $t_k\ge k\delta_k$, we have 
\[
\|\Phi^*(\Phi w_k)\|\le \frac{c\delta_k}{t_k}\|v_k\|+\|\Phi^*\|\frac{\|y_k-y_0\|}{t_k}\le \frac{c\delta_k}{t_k}\|v_k\|+\|\Phi^*\|\frac{\delta_k}{t_k}\le \frac{1}{k}(cL+\|\Phi^*\|)\to 0,
\]
which yields $\Phi^*(\Phi w)=0$, i.e., $\Phi w=0$. Hence, we have 
\[
w\in \Ker \Phi\cap T_{\partial R^*({\rm Im}\, \Phi^*)}(x_0). 
\]
This contradicts  \eqref{con:SOSC}, as $\|w\|=1$.  The proof is complete.
\end{proof}

From the proof of the above theorem, solution uniqueness is necessary for stable recovery; see also \cite[Proposition~3.1]{FNT23} for a similar observation. The tangent cone in \eqref{con:SOSC} gives us a geometric condition for stable recovery. Conceptually, it is  a second-order condition, as the tangent cone acting on the image of the subdifferential mapping is akin to taking another ``derivative''.  This is totally different from other sufficient conditions for stable recovery \cite{CRPW12, GSH11, FPVDS13, FNT23,VPF15} that only use first-order information of the regularizer. In the rest of the section, we will establish several consequences of this theorem, compare them with other known results on stable recovery, and also build up some examples to distinguish  our result from others.

From the above point of view, it is not surprising that non-sharp optimal solutions can be stably recoverable, as in  Example~\ref{ex:l12}. It is also natural to think that stable recovery is more associated with strong minima. But the next example shows that strong minima are insufficient to ensure stable recovery. 
\begin{Example}[Failure of stable recovery under strong minima]\label{ex:NoSR}\rm
Consider the following $\ell_1/\ell_2$ optimization problem with two groups:
\begin{equation}\label{p:E12}
\min_{x\in \mathbb{R}^4} R(x) = \sqrt{x_1^2+x_2^2}+ \sqrt{x_3^2+x_4^2}, \quad \text{subject to} \quad \Phi x = y_0   
\end{equation}
with $\Phi=\begin{pmatrix} 1&0&0&-1\\0&1&0&1\\0&0&1&0\end{pmatrix} 
$,  $x_0=(0,1,0,0)^T$, and $y_0=(0,1,0)^T$.  Note that $\Ker \Phi=\R(1,-1,0,1)^T$. Any feasible point $x$ takes the form $ x_0 + t(1,-1,0,1)^T$ for $t \in \mathbb{R}$. With the function $\ph$ defined in \eqref{eq:ph}, we have 
\begin{equation*}
    \ph(x) - \ph(x_0) = \sqrt{t^2 + (1-t)^2} + |t| - 1 = \frac{t^2 + 2(|t|-t)}{\sqrt{t^2 + (1-t)^2} + 1 - |t|}\in
  \left[\frac{\|x-x_0\|^2}{9},\frac{\|x-x_0\|^2}{3}\right]
\end{equation*}
when $x$ is close to $x_0$. Thus $x_0$ is a strong (unique) optimal solution of the function $\ph$ or problem \eqref{p:E12}, but it is not a sharp one (by choosing $t>0$). 

Next, we claim that $\R_+(1,-1,0,1)\subset\Ker \Phi\cap T_{\partial R^*({\rm Im}\, \Phi^*)}(x_0)$. For any $w:= (a,-a,0,a)$ with $a>0$ and a sequence $t_k\dn 0$, we build $w_k:=(a,-a,c_k,d_k)$ with
\[
c_k:=\dfrac{a\sqrt{2t_ka-2t_k^2a^2}}{\|(t_ka,1-t_ka)\|} \quad\mbox{and} \quad d_k:=\dfrac{a(1-2t_ka)}{\|(t_ka,1-t_ka)\|}.
\]
Observe that $w_k \to w$ and $\|(c_k,d_k)\|=a\neq0$. Moreover,  we have
\[
v_k:=\nabla R(x_0+t_kw_k)= \left(\dfrac{t_ka}{\|(t_ka,1-t_ka)\|},\dfrac{1-t_ka}{\|(t_ka,1-t_ka)\|},\dfrac{\sqrt{2t_ka-2t_k^2a^2}}{\|(t_ka,1-t_ka)\|},  \dfrac{1-2t_ka}{\|(t_ka,1-t_ka)\|}\right).
\]
It follows that $x_0+t_kw_k\in \partial R^*(v_k)$. Note also that  $v_k\in \Im \Phi^*$. Hence, $x_0+t_kw_k\in \partial R^*(\Im \Phi^*)$, which helps us to conclude that $w\in T_{\partial R^*({\rm Im}\, \Phi^*)}(x_0)$. By Theorem~\ref{thm:Rob}, $x_0$ is not stably recoverable, although it is a strong and unique optimal solution. $\hfill{\triangle}$
\end{Example}
Next, let us  analyze condition~\eqref{con:SOSC} for different classes of regularizers.  
Recall here the {\em critical cone} of $R$ at $x_0$ defined by
\begin{equation}\label{def:Cg}
C_R(x_0):=\{w\in \XX|\; dR(x_0|\,w)\le 0\}.
\end{equation}
 This cone is important for us to estimate for the set appeared in \eqref{con:SOSC}. As $R$ is a continuous convex function, it is locally Lipschitz at $x_0$. Hence, $\partial R(x_0)$ is a compact set. It follows from the formula \eqref{eq:DDmax} that $dR(x_0|\,\cdot)$ is a continuous function. Hence, the interior and the boundary of this critical cone are determined respectively by
\begin{equation}\label{eq:Inbd}
    {\rm int}\, C_R(x_0):=\{w\in \XX|\; dR(x_0|\,w)< 0\}\quad \mbox{and}\quad \bd C_R(x_0):=\{w\in \XX|\; dR(x_0|\,w)=0\}. 
\end{equation}
When $x_0$ is a minimizer of problem~\eqref{p:P0}, for any $w\in \Ker \Phi$, we have $\Phi(x_0+tw)=y_0$ and thus
\[
 dR(x_0|\,w)=\lim_{t\dn 0}\frac{R(x_0+tw)-R(x_0)}{t}\ge 0. 
\]
It follows that  
\begin{equation}\label{eq:Inbd2}
    \Ker \Phi\cap {\rm int}\, C_R(x_0)=\emptyset\quad \mbox{and}\quad  0\in \Ker \Phi\cap\bd C_R(x_0)\neq \emptyset. 
\end{equation}

\begin{Proposition}\label{prop:Est} Suppose that $R$ is a continuous convex function and $x_0$ is an optimal solution of problem~\eqref{p:P0}. Then  $\partial R^*(\Im \Phi^*)$ is a closed set of $\XX$. Moreover, we have 
\begin{equation}\label{eq:Inc}
 \bigcup_{v\in \Delta(x_0)} T_{\partial R^*(v)}(x_0)\subset T_{\partial R^*({\rm Im}\, \Phi^*)}(x_0)\subset \bigcup_{v\in \Delta(x_0)}N_{\partial R(x_0)}(v).
\end{equation}
Consequently, 
\begin{equation}\label{eq:Inc2}
\bigcup_{v\in \Delta(x_0)} \left(\Ker\Phi \cap  T_{\partial R^*(v)}(x_0) \right)\subset \Ker \Phi\cap T_{\partial R^*({\rm Im}\, \Phi^*)}(x_0)\subset \Ker \Phi \cap \bd C_R(x_0). 
\end{equation}
\end{Proposition}
\begin{proof} To prove the closedness of $\partial R^*(\Im \Phi^*)$, pick any sequence $x_k\in \partial R^*(\Im \Phi^*)$ converging to some $\bar x$. We find $v_k\in \Im \Phi^*$ such that $x_k\in \partial R^*(v_k)$, which means $v_k\in \partial R(x_k)$. Since $R$ is a continuous convex function, it is Lipschitz continuous with some modulus $L>0$ around $\bar x$. It follows that  $\|v_k\|\le L$ for sufficiently large $k$. By passing to a subsequence, we may suppose that $v_k\to \bar v\in \Im \Phi^*$. Since the graph of the subdifferential mapping $\partial R$ is closed, we have  $\bar v\in \partial R(\ox)$, i.e., $\ox\in \partial R^*(\bar v)\subset \partial R^*(\Im \Phi^*)$. This verifies the closedness of $\partial R^*(\Im \Phi^*)$. 

The left inclusion  in \eqref{eq:Inc} is trivial due to the definition of $\Delta(x_0)$ in \eqref{eq:DC}. To show the right inclusion in \eqref{eq:Inc}, pick any $w\in T_{\partial R^*({\rm Im}\, \Phi^*)}(x_0)$. There exist sequences $t_k\dn 0$, $w_k\to w$, and $v_k\in \Im \Phi^*$ such that $x_0+t_kw_k\in \partial R^*(v_k)$, which means $v_k\in \partial R(x_0+t_kw_k)$. It is similar to the proof of the first part, we may suppose that $v_k\to v\in \Im \Phi^*\cap \partial R(x_0)=\Delta(x_0)$.   As $\partial R$ is a monotone mapping, for any $z\in \partial R(x_0)$, the latter gives us that 
\[
0\le\la v_k-z,x_0+t_kw_k-x_0\ra=t_k\la v_k-z,w_k\ra.
\]
It follows that $\la v_k-z,w_k\ra\ge 0$. By taking $k\to \infty$, we have $
\la z-v,w\ra\le 0$ for any $z\in \partial R(x_0),$
which clearly implies that $w\in N_{\partial R(x_0)}(v)$ due to \eqref{eq:Nor} and thus verifies \eqref{eq:Inc}.

The left inclusion in \eqref{eq:Inc2} is straightforward from \eqref{eq:Inc}. Let us prove the right inclusion in \eqref{eq:Inc2}. Pick any $w\in \Ker \Phi\cap T_{\partial R^*({\rm Im}\, \Phi^*)}(x_0)$. It follows from \eqref{eq:Inc} that there exists some $v\in \Delta(x_0)$ such that $w\in N_{\partial R(x_0)}(v)$. For any $z\in \partial R(x_0)$, we have
\[
\la z,w\ra\le \la v,w\ra=0,
\]
as $v\in \Im \Phi^*$ and $w\in \Ker\Phi$. This together with \eqref{eq:DDmax} and \eqref{eq:Inbd} implies that $w\in \bd C_R(x_0)$. The right inclusion in \eqref{eq:Inc2} is verified. 
\end{proof}

The following result recovers \cite[Theorem~3.10]{FNT23} in our framework, at which it is shown that sharp minima are sufficient for stable recovery. Other results about stable recovery such as \cite[Theorem~4.7]{GSH11}, \cite[Proposition~2.2]{CRPW12}, \cite[Theorem~2]{FPVDS13}, and  \cite[Theorem~2]{VPF15} are  established under some  conditions that are equivalent to sharp minima too. Consequently, our characterization in Theorem~\ref{thm:Rob} provides a weaker condition for stable recovery for the case of general convex regularizers.

\begin{Corollary}[Sharp minima for stable recovery]\label{cor:Sharp} Suppose that $R$ is a continuous convex function. If  $x_0$ is a sharp minimizer of problem \eqref{p:P0} then condition \eqref{con:SOSC} holds, i.e., stable recovery occurs at $x_0$. 
\end{Corollary}
\begin{proof}
    Suppose $x_0$ is a sharp solution. \cite[Proposition~3.8]{FNT23} tells us that $\Ker \Phi \cap C_R(x_0) = \{0\}$. This together with \eqref{eq:Inc2} verifies condition~\eqref{con:SOSC}. By Theorem~\ref{thm:Rob}, $x_0$ is stably recoverable.
\end{proof}

As condition~\eqref{con:SOSC} is a second-order type, while sharp minima can be characterized by only  using first-order information \cite{P87,HKS23,FNT23}, we  expect the difference between stable recovery and sharp minima in general; see also  our Example~\ref{ex:l12}. The following result shows that strong minima imply stable recovery, when the regularizer $R(x)$ is twice differentiable.  Although smooth regularizers are not the main case in our paper, this corollary sheds  light on the impact of strong minima in stable recovery. 

\begin{Corollary}[Strong minima and stable recovery with smooth regularizers]\label{cor:stro} Suppose that $x_0$ is a minimizer of problem \eqref{p:P0} and  $R$ is a convex function that is continuously twice  differentiable around  $x_0$. Then we have
\begin{equation}\label{eq:Inc3}
\Ker \Phi\cap T_{\partial R^*({\rm Im}\, \Phi^*)}(x_0)\subset \Ker \Phi\cap \Ker \nabla^2 R(x_0). 
\end{equation}
Consequently, stable recovery occurs at $x_0$ provided that 
\begin{equation}\label{con:SCD}
\Ker \Phi\cap \Ker \nabla^2 R(x_0)=\{0\},
\end{equation}
which is equivalent to the strong minimum at  $x_0$ of problem~\eqref{p:P0}. 
\end{Corollary}
\begin{proof} To justify \eqref{eq:Inc3},  pick any $w\in \Ker \Phi\cap T_{\partial R^*({\rm Im}\, \Phi^*)}(x_0)$. There exist sequences $t_k\dn 0$, $w_k\to w$, and $v_k\in \Im \Phi^*$ such that $x_0+t_kw_k\in \partial R^*(v_k)$. Since $R$ is continuously twice differentiable around  $x_0$, we have
\[
v_k=\nabla R(x_0+t_kw_k)=\nabla R(x_0)+t_k\nabla^2 R(x_0)w_k+o(t_k). 
\]
Since $x_0$ is an optimal solution of problem \eqref{p:P0}, $0\in \nabla R(x_0)+\Im \Phi^*$, i.e., $\nabla R(x_0)\in \Im \Phi^*$. As $v_k\in \Im \Phi^*$ and $w\in \Ker \Phi$, we obtain from the above equality that 
\[
0=\la v_k,w\ra=\la t_k\nabla^2 R(x_0)w_k+o(t_k),w\ra,
\]
which implies $
0=\la \nabla^2 R(x_0)w_k,w\ra+o(t_k)/t_k$.
By letting $k\to \infty$, this certainly gives us that $\la \nabla^2 R(x_0)w,w\ra=0$. Since $R$ is a convex function, $\nabla^2 R(x_0)$ is positive semi-definite. It follows that  $w\in \Ker \nabla^2 R(x_0)\cap \Ker \Phi$, which verifies \eqref{eq:Inc3}.

If condition \eqref{con:SCD} holds, \eqref{con:SOSC} is satisfied due to \eqref{eq:Inc3}. By Theorem~\ref{thm:Rob}, $x_0$ is stably recoverable. It remains to show that $x_0$ is a strong optimal solution  of \eqref{p:P0} if and only if condition \eqref{con:SCD} holds. Indeed, $x_0$ is a strong minimizer of \eqref{p:P0} if and only if it is the strong minimizer of the function $\ph$ defined in \eqref{eq:ph}. By Lemma~\ref{Fa} and the sum rule \eqref{eq:sum}, the latter is equivalent to
\begin{equation*}
    \lange w, \nabla^2 R(x_0)w \rangle + d^2 \delta_{\Phi^{-1}(y_0)}(x_0|\,-\nabla R(x_0))(w)=d^2 \ph(x_0|\,0)(w) > 0 \quad \text{for all} \quad w \ne 0.
\end{equation*}
By using \eqref{eq:SD}, it is easy to verify that   $d^2\delta_{\Phi^{-1}(y_0)}(x_0|-\nabla R(x_0))(w)=\delta_{{\rm Ker}\, \Phi}(w).$ The above condition is equivalent to \eqref{con:SCD}, since $\nabla^2 R(x_0)$ is positive semi-definite. 
\end{proof}

In the classical Morozov/Tikhonov regularization method, the regularizer $R(x)$ is chosen as $\frac{1}{2}\la Qx,x\ra$ for some positive semi-definite matrix $Q\in \R^{n\times n}$. By Corollary~\ref{cor:stro}, stable recovery occurs at an optimal solution $x_0$ of problem~\eqref{p:P0} provided that
\[
\Ker\Phi\cap \Ker Q=\{0\},
\]
which is always the case when $Q$ is a positive definite matrix. An important case of this class is $R(x)=\|\nabla x\|^2$, where $\nabla$ is the {\em discrete gradient operator} \cite{ROF92,CP11} frequently used in imaging for denoising/deblurring signals/ images; see also \eqref{eq:TV} and \eqref{eq:total}. Note that $\nabla^2 R(x) = 2\nabla^{*}\nabla$, which is positive semi-definite. This implies that  $\Ker \nabla^2R(x) = \Ker \nabla^*\nabla = \Ker \nabla$. Hence, condition \eqref{con:SCD} turns into $\Ker\Phi\cap \Ker \nabla^*\nabla =\{0\}$, which is equivalent to 
\[
\Ker\Phi\cap \Ker \nabla =\{0\}.
\]
It is easy to check that  $\dim(\Ker \nabla)=1$. The above condition is very likely valid when $\Phi$ is some blurring or random  matrix. 

Next we show  that stable recovery is actually equivalent to solution uniqueness for the broad class of convex {\em piecewise linear-quadratic} regularizers. This is a far-reaching extension of \cite[Theorem~4.7]{GSH11} for the $\ell_1$ regularized problems and a huge improvement of \cite[Corollary~3.15]{FNT23} that only established the convergence rate $\delta^\frac{1}{2}$ of $\|x_\mu-x_0\|$ for this class. Recall that the function $R:\XX\to \R$ is said to be  convex piecewise linear-quadratic if its domain can be written as the union of finitely many polyhedral sets and for each of those sets, $R$ is represented as $\frac{1}{2}\la Ax,x\ra+\la b,x\ra+c$ with some $A\in \mathscr{L}(\XX,\XX)$ being a positive semi-definite operator, $b\in \XX$, and $c\in \R$. This class contains several important regularizers in optimization and linear inverse problems such as the elastic net \cite{ZH05}, Huber norm \cite{H73}, and the discrete Blake-Zisserman regularizer \cite{BZ87}. It is worth noting that in this case, the function $
\ph$ in \eqref{eq:ph} is also convex  piecewise linear-quadratic. A unique optimal solution is actually the strong optimal solution because convex  piecewise linear-quadratic functions satisfy the so-called {\em quadratic growth condition}; see, e.g.,  \cite[the proof of Corollary~4.2]{BLN21}

\begin{Corollary}[Stable recovery and strong minima]\label{cor:Smooth} Let $R$ be a continuous convex piecewise linear-quadratic function and $x_0$ be a solution of problem~\eqref{p:P0}. Stable recovery occurs at $x_0$ if and only if $x_0$ is a unique (strong) optimal solution of problem~\eqref{p:P0}.
\end{Corollary}
\begin{proof}
As noted after Theorem~\ref{thm:Rob} that solution uniqueness at $x_0$ is necessary for stable recovery at $x_0$. It remains to show that if $x_0$ is the unique optimal solution of problem~\eqref{p:P0},  it is stably recoverable. Note from \cite[Theorem~11.14 and Proposition~12.30]{RW98} that  the conjugate function $R^*$ is also piecewise linear-quadratic and  the graph of $\partial R^*$ is {\em piecewise polyhedral} in the sense that its graph, $\gph\partial R^*:=\{(v,x)\in \XX\times \XX|\, x\in \partial R^*(v)\}$, is a union of finitely many polyhedrons. 
Let us consider the projection   $P_{\XX}:\XX\times \XX$ with $P_\XX(v,x)=x$ for any $(v,x)\in \XX\times\XX$. Note that $\gph \partial R^* \cap (\Im \Phi^* \times \XX)$ is also a union of finitely many polyhedrons. So is its projection $P_{\XX}(\gph \partial R^* \cap (\Im \Phi^* \times \XX))=\partial R^* (\Im\Phi^*)$. Hence, we may write 
\begin{equation}\label{eq:Poly}
\partial R^*(\Im \Phi^*) = \bigcup_{k=1}^m P_k,
\end{equation}
where $P_k$, $k=1, \ldots, m$ are some (closed) polyhedra. Set $\mathcal{I}:=\{k\in \{1,\ldots,m\}|\; x_0\in P_k\}$,  the left-hand side of \eqref{con:SOSC} turns into 
\begin{equation*}
\begin{aligned}
        \Ker \Phi\cap T_{\partial R^*({\rm Im}\, \Phi^*)}(x_0)
        =\Ker \Phi \cap \left[\underset{k\in \mathcal{I}}{\bigcup}T_{P_k}(x_0)\right]=\underset{k\in \mathcal{I}}{\bigcup}\left(\Ker \Phi \cap T_{P_k}(x_0)\right).
    \end{aligned}
    \end{equation*}
Since each $P_k-x_0$ is a polyhedron containing $0$, its conic hull $\cone(P_k-x_0)$ is closed for $k\in \mathcal{I}$. It follows from \eqref{eq:TanC} that
\[
\Ker \Phi\cap T_{\partial R^*({\rm Im}\, \Phi^*)}(x_0)=\underset{k\in \mathcal{I}}{\bigcup}\left(\Ker \Phi \cap \cone(P_k-x_0)\right).
\]
Pick any $w\in \Ker \Phi\cap T_{\partial R^*({\rm Im}\, \Phi^*)}(x_0)$. The above expression deduces the existence of $k\in\mathcal{I}$ such that $w\in \cone (P_{k}-x_0)$. We may find some $t>0$ satisfying  $x_0+tw \in P_{k} \subset \partial R^*(\Im \Phi^*)$ by \eqref{eq:Poly}. Therefore, there exists $v \in \Im \Phi^* $ such that $x_0+tw \in \partial R^*(v)$, which means $v \in \partial R(x_0+tw)$ and thus $v\in \Delta(x_0+tw)$. On the other hand, $\Phi (x_0+tw) = y_0$ since $w\in \Ker \Phi$. It follows that $x_0 + tw$ is another optimal solution of problem \eqref{p:P0}.  Since $x_0$ is the unique solution of \eqref{p:P0}, we have $w=0$. This verifies condition~\eqref{con:SOSC} and shows that $x_0$ is stably recoverable by Theorem~\ref{thm:Rob}.
\end{proof}

A subclass of piecewise linear-quadratic functions is the class of  {\em piecewise-linear} functions that include a lot of regularizers used for inverse linear problems such as the $\ell_1$ norm, the 1D total variation and 2D anisotropic total variation seminorms, and the fussed Lasso. Consequently, the conclusion of Corollary~\ref{cor:Smooth} is also valid, but unique optimal solutions of problem~\eqref{p:P0} are actually  sharp minimizers for this class. It recovers \cite[Theorem~3.10]{FNT23} for this case and \cite[Theorem~4.7]{GSH11} for the $\ell_1$ norm.

The main challenge of Theorem~\ref{thm:Rob} is the computation of the tangent cone $T_{\partial R^*({\rm Im}\, \Phi^*)}(x_0)$. The set $\partial R^*({\rm Im}\, \Phi^*)$ is highly intricate and possibly non-convex.  Even in the simple case that  $R$ is continuously twice differentiable, it is not clear to us how to calculate it fully; see \eqref{eq:Inc3} in Corollary~\ref{cor:stro}. In the next section, we focus on understanding the condition \eqref{con:SOSC} for analysis $\ell_1/\ell_2$ regularization problems, which play crucial roles in many statistical models with group sparsity \cite{YL06} and imaging \cite{ROF92,CP11}.

\section{Stable recovery for analysis group sparsity  regularized linear inverse problems}    
\setcounter{equation}{0}

An important convex regularizer widely used in optimization is the $\ell_1/\ell_2$ norm
(also known as the group Lasso regularizer) defined by 
\begin{equation}\label{eq:l12}
    \|u\|_{1,2}:=\sum_{J\in \mathcal{J}}\|u_J\|\quad \mbox{for any}\quad u\in \R^{p},
\end{equation}
where $\mathcal{J}$ is a collection of index sets partitioning $\{1, \ldots,p\}$ into $q$ different groups, for any $J\in \mathcal{J}$, $u_J\in \R^{|J|}$ is the component vector of $u$ with index $J$, and $\|u_J\|$ is the standard Euclidean norm in $\R^{|J|}$ with $|J|$ being the cardinality of $J$. Let $D$ be a known $n\times p$ matrix and $R:\R^n\to \R$ be the {\em analysis}  $\ell_1/\ell_2$ seminorm defined by 
\begin{equation}\label{norm:l12}
    R(x):=\|D^*x\|_{1,2}=\sum_{J\in \mathcal{J}}\|D^*_Jx\|\quad \mbox{for}\quad x\in \R^n,
\end{equation}
where $D_J$ is the $n\times |J|$ containing  all the $J$-index {\em columns} from $D$. 
Problem \eqref{p:P0} turns to
\begin{equation}\label{p:P12}
\min_{x\in \R^n}\quad \|D^*x\|_{1,2}\quad \mbox{subject to}\quad \Phi x=y_0,
\end{equation}
where $\Phi:\R^n\to \R^m$ is a linear operator and $y_0\in \R^m$. Two important cases of $D$ that will be addressed in our paper are (a) $D^*=\Id:\R^n\to \R^n$, the identity matrix that has been used significantly in statistics \cite{YL06} to promote the group sparsity of solutions of problem~\eqref{p:P12}; (b) $D^*=\nabla:\R^{n_1\times n_2}\to \R^{2n_1n_2}$, 
the two-dimensional (2D) {\em discrete gradient operator} defined by
\begin{equation}\label{eq:TV}
    (\nabla x)_{i,j}:=\begin{pmatrix}
        (\nabla x)^1_{i,j}\\(\nabla x)^2_{i,j}
    \end{pmatrix}\quad \mbox{for}\quad x\in \R^{n_1\times n_2}\quad \mbox{with}
\end{equation}
\begin{equation}\label{eq:total}
   (\nabla x)^1_{i,j}:=\left\{\begin{array}{ll}x_{i+1,j}-x_{i,j}\quad &\mbox{for}\quad i < n_1\\ 
   0 \quad &\mbox{for}\quad i = n_1,
\end{array}\right.\; \mbox{and}\;\;\;\\
(\nabla x)^2_{i,j}:=\left\{\begin{array}{ll}x_{i,j+1}-x_{i,j}\quad &\mbox{for}\quad j < n_2\\ 
0 \quad &\mbox{for}\quad j= n_2.
\end{array}\right.
\end{equation}
This regularizer converts problem \eqref{p:P12} to the so-called {\em isotropic total variation problem} that is widely used in imaging \cite{ROF92,CP11}. 

Let $x_0$ be an optimal solution of problem~\eqref{p:P12} and  define $\oy:=D^*x_0$. Define the {\em active index}  and {\em inactive index} set of $D^*x_0$ respectively by
\begin{equation}\label{def:ind}
\mathcal{I}:=\{J\in \mathcal{J}|\; D^*_Jx_0=\oy_J\neq 0\}\quad\mbox{and}\quad \mathcal{I}^c:=\mathcal{J}\setminus\mathcal{I}=\{J\in \mathcal{J}|\;  D^*_Jx_0=\oy_J=0\}.   
\end{equation}
 Define further that the vector $e\in \R^p$ by 
\begin{equation}\label{def:e}
    e:=\left\{\begin{array}{ll} e_J=\frac{\oy_J}{\|\oy_J\|}\; &\mbox{if}\; J\in \mathcal{I}\\
    e_J=0\; &\mbox{if}\; J\in \mathcal{I}^c. \end{array}\right.
\end{equation}
 Note that the function $\|u\|$, $u\in \R^{|J|}$ is twice differentiable around any $u\neq 0$ with the Hessian 
\begin{equation}\label{eq:Hes1}
\nabla^2 \|u\|=\nabla \left(\frac{u}{\|u\|}\right)=\frac{1}{\|u\|}\Id_J -\frac{1}{\|u\|^3}uu^T, 
\end{equation}
where $\Id_J$ is the identity mapping in $\R^{|J|}$. Throughout the section, we define the following matrix 
\begin{equation}\label{def:AJ}
A_J:=\nabla^2 \|\oy_J\|=\frac{1}{\|\oy_J\|}\Id_J -\frac{1}{\|\oy_J\|^3}\oy_J(\oy_J)^T \quad \mbox{for any}\quad J\in \mathcal{J}.
\end{equation}
Moreover, the linear operator $A_\mathcal{I}:\R^\mathcal{I}\to \R^\mathcal{I}$ is defined by 
\begin{equation}\label{def:AI}
A_\mathcal{I}w:=(A_Jw_J)_{J\in \mathcal{I}}\quad \mbox{for all}\quad w\in \R^{\mathcal{I}}.
\end{equation}
Here the notion $\R^\mathcal{I}$ is just like $\R^{|\mathcal{I}|}$, but writing $\R^\mathcal{I}$ allows us to match the index $\mathcal{I}$ when picking any $w=(w_J)_{J\in \mathcal{I}}\in \R^\mathcal{I}$.
For an index set $\mathcal{K}\subset \mathcal{J}$, we denote the {\em sphere} $\mathbb{S}_\mathcal{K}$ and the {\em unit ball} $\B_\mathcal{K}$ by
\begin{equation}\label{def:SB}
   \mathbb{S}_\mathcal{K}:=\{w\in \R^\mathcal{K}|\; \|w_J\|=1, J\in \mathcal{K}\}\quad \mbox{and}\quad  \mathbb{B}_\mathcal{K}:=\{w\in \R^\mathcal{K}|\; \|w_J\|\le1, J\in \mathcal{K}\}.
\end{equation}
Observe that 
\begin{equation}\label{eq:Sub}
    \partial \|\oy\|_{1,2}=e_{\mathcal{I}}\times \B_{\mathcal{I}^c}.
\end{equation}
By \eqref{eq:DDmax} and \eqref{eq:Inbd}, the boundary of the critical cone  $C_R(x_0)$ in \eqref{def:Cg} is deduced to 
\begin{equation}\label{eq:CCR}
\bd C_R(x_0) = \{ w \in \R^n|\; \la 
e_\mathcal{I},D_\mathcal{I}^*w \ra + \|D^*_{\mathcal{I}^c}w\|_{1,2} = 0\}.
\end{equation}
Note that the set of all dual certificates  \eqref{eq:DC} in this case is 
\[
\Delta(x_0)= D (\partial \|\oy\|_{1,2}) \cap \Im \Phi^*.
\]
The following set is important in the analysis of our  paper
\begin{equation}\label{def:Wx}
    W(x_0):=\Ker \Phi\cap \bd C_R(x_0).
\end{equation}
As discussed in Proposition~\ref{eq:Est} and Corollary~\ref{cor:Sharp}, $W(x_0)=0$ if and only if $x_0$ is a sharp optimal solution; see, e.g., \cite[Proposition~3.8]{FNT23}. It is more interesting to study stable recovery when $W(x_0)\neq \{0\}$. The following result gives a representation of this set.

\begin{Proposition}\label{prop:Crit} Suppose that $x_0$ is a minimizer of problem \eqref{p:P12}. For any $v\in \partial \|\oy\|_{1,2}$ satisfying $Dv \in \Delta(x_0)$,  define the following two index sets
\begin{equation}\label{def:KH}
    \mathcal{K}:=\{J\in  \mathcal{I}^c|\; \|v_J\|=1\}\quad \mbox{and} \quad\mathcal{H}:= \mathcal{I}^c\setminus\mathcal{K}=\{J\in  \mathcal{I}^c|\; \|v_J\|<1\}. 
\end{equation}
Then we have 
\begin{equation}\label{eq:KCrit}
     W(x_0)=\{w\in \Ker\Phi|\; D^*_J w\in \R_+v_J, J\in \mathcal{K}\mbox{ and } D^*_Jw=0, J\in \mathcal{H}\}. 
\end{equation}
\end{Proposition}
\begin{proof}
As $v\in \partial \|\oy\|_{1,2}$, it follows from \eqref{eq:Sub} that $v_\mathcal{I}=e_\mathcal{I}$ and $v_{\mathcal{I}^c}\in \B_{\mathcal{I}^c}$.
For any $w \in \Ker \Phi$, since $Dv\in \Im \Phi^*$, we have
    \begin{eqnarray*}
\begin{aligned}
            \la 
D_\mathcal{I}e_\mathcal{I},w \ra + \|D^*_{\mathcal{I}^c}w\|_{1,2}&=\disp \sum_{J\in \mathcal{I}}\la D_Je_J,w\ra+\sum_{J\in \mathcal{I}^c}\|D^*_{J}w\|\\
&= \disp \la Dv,w \ra - \sum_{J\in \mathcal{I}^c} \la D_Jv_J,w \ra +\sum_{J\in \mathcal{I}^c}\|D^*_{J}w\|\\
            &= \disp - \sum_{J\in \mathcal{I}^c} \la v_J,D^*_{J}w \ra + \sum_{J\in \mathcal{I}^c}\|D^*_{J}w\| \\
            &= \disp \sum_{J \in \mathcal{K}} \left(\|D^*_{J}w\| - \la v_J, D^*_{J}w\ra\right) + \sum_{J \in \mathcal{H}} \left(\|D^*_{J}w\| - \la v_J, D^*_{J}w\ra\right)\\
            &\ge 0 \qquad \text{(as}\;\; v_{\mathcal{I}^c}\in \B_{\mathcal{I}^c}\text{)}.
\end{aligned}
    \end{eqnarray*}    
Due to \eqref{eq:CCR} and \eqref{def:Wx}, $w \in W(x_0)$ if and only if $w\in \Ker\Phi $ and  there exist $\lambda_J \geq 0$,  $J \in \mathcal{K}$ such that $D^*_{J}w = \lambda_J v_J$ and $D^*_{J}w =0$ for any $J \in \mathcal{H}$. This clearly verifies \eqref{eq:KCrit}.
\end{proof}

Stable recovery or linear convergence rate at $x_0$ for the $\ell_1/\ell_2$ problem \eqref{p:P12} was studied in \cite{G11,H13} in infinite dimensional settings. By reducing these works to finite dimensional frameworks,  we observe that \cite[Proposition~6.1 and Proposition~7.1]{G11} and \cite[Theorem I.3]{H13} used the following the {\em Restricted Injectivity} condition  
\begin{equation}\label{con:G11}
    \{w\in \Ker\Phi|\; D^*_Jw=0, J\in \mathcal{H}\}=\{0\}
\end{equation}
to ensure stable recovery at $x_0$. By \eqref{eq:KCrit}, this condition implies that $W(x_0)=0$, i.e., $x_0$ is a sharp minimizer. This means  sharp minima are still behind \cite{G11,H13} for stable recovery. In this section, we will derive some new conditions that are totally independent of sharp minima. Before going there, let us define the set
\begin{equation}\label{def:ME}
    \mathcal{E}:=\{w\in \R^n|\; D^{*}_{J}w \in \R\{\oy_J\}, J\in \mathcal{I}\},
\end{equation}
which is introduced in \cite[Theorem~5.1]{FNT23} to characterize the solution uniqueness of problem~\eqref{p:P12}. The following result recalls that equivalence between solution uniqueness and strong minima of problem~\eqref{p:P12} and its characterization via the set \eqref{def:ME} above from \cite[Theorem~5.1 and Theorem~5.3]{FNT23}. 
\begin{Theorem}[Equivalence between solution uniqueness and strong minima]\label{thm:Equi} Let $x_0$ be a minimizer of problem~\eqref{p:P12}. The following are equivalent: 
\begin{itemize}
\item[{\bf (i)}] $x_0$ is the unique solution of problem~\eqref{p:P12}.

    \item[{\bf (ii)}] $x_0$ is the strong solution of problem~\eqref{p:P12}. 

    \item[{\bf (iii)}] $W(x_0)\cap \mathcal{E}=\{0\}$ with $\mathcal{E}$ being defined in \eqref{def:ME}. 
 \end{itemize}   
\end{Theorem}
\begin{proof} Since $x_0$ is a minimizer, \eqref{eq:Inbd2} tells us that $\Ker\Phi \cap{\rm int}\, C_R(x_0)=\emptyset$. The equivalence between {\bf (i)}, {\bf (ii)}, and {\bf (iii)} follows directly from~\cite[Theorem~5.1 and Theorem~5.3]{FNT23}. \end{proof}

We are ready to compute the set in \eqref{con:SOSC} for problem~\eqref{p:P12}.

\begin{Theorem}[Stable recovery for analysis group sparsity regularized linear inverse problems]\label{thm:Full}
We have 
\begin{eqnarray}
\begin{aligned}\label{eq:Tg}
\Ker \Phi \cap T_{\partial R^*({\rm Im}\, \Phi^*)}(x_0) \subset &\left\{w\in W(x_0)|\; D_\mathcal{I}A_{\mathcal{I}}D^*_{\mathcal{I}}w  \in  T_{ \left({\rm Im}\,\Phi^*+D_{\mathcal{I}^c}\left(\mathbb{S}_{\mathcal{K}(w)}\times\mathbb{B}_{\mathcal{H}(w)}\right)\right)\cap D_{\mathcal{I}}(\mathbb{S}_{\mathcal{I}})}(D_{\mathcal{I}}e_\mathcal{I})\right\}, 
\end{aligned}
\end{eqnarray}
where the index sets $\mathcal{K}(w)$ and $\mathcal{H}(w)$ are defined respectively by
\begin{equation}\label{eq:KHw}
  \mathcal{K}(w):=\{J\in \mathcal{I}^c|\; D^{*}_Jw\neq 0\}\quad\mbox{and}\quad   \mathcal{H}(w):=\{J\in \mathcal{I}^c|\; D^{*}_Jw = 0\}.
\end{equation}
Consequently, stable recovery occurs at $x_0$ for problem~\eqref{p:P12} when the following condition is satisfied 
\begin{equation}\label{con:full}
   \left\{w\in W(x_0)|\; D_\mathcal{I}A_{\mathcal{I}}D^*_{\mathcal{I}}w  \in  T_{ \left({\rm Im}\,\Phi^*+D_{\mathcal{I}^c}\left(\mathbb{S}_{\mathcal{K}(w)}\times\mathbb{B}_{\mathcal{H}(w)}\right)\right)\cap D_{\mathcal{I}}(\mathbb{S}_{\mathcal{I}})}(D_{\mathcal{I}}e_\mathcal{I}) \right\}=\{0\}.
    \end{equation}
If, additionally, $D^*$ is surjective, the inclusion ``$\subset$'' in condition \eqref{eq:Tg} turns into equality. In this case, condition~\eqref{con:full} is also necessary for stable recovery at $x_0$.
\end{Theorem}

\begin{proof} To prove the inclusion ``$\subset$'' in \eqref{eq:Tg}, pick any $w\in \Ker \Phi \cap T_{\partial R^*({\rm Im}\, \Phi^*)}(x_0)$. By Proposition~\ref{prop:Est}, $w\in W(x_0)$. By the definition of the tangent cone, there exist sequences $t_k\dn 0$, $w_k\in \R^n$, $v_k\in \Im \Phi^*$ such that $x_0+t_kw_k\in \partial R^*(v_k)$ and $w_k\to w$. As $v_k\in \partial R(x_0+t_kw_k)$, we have 
\begin{equation}\label{eq:vJ}
\begin{aligned}
    v_k &\in D(\partial \| \cdot \|_{1,2}(D^*(x_0+t_kw_k)))=D\left(\prod_{J\in \mathcal{J} }\partial \| \cdot \|(D_J^*(x_0+t_kw_k)) \right)\\
    &\subset \sum_{J \in \mathcal{I}}D_J\left( \frac{D^*_J(x_0+t_kw_k)}{\|D^*_J(x_0+t_kw_k)\|}\right)+\sum_{J \in \mathcal{K}(w)}D_J\left( \frac{D^*_Jw_k}{\|D^*_Jw_k\|}\right)+\sum_{J \in \mathcal{H}(w)}D_J\mathbb{B}_J.
\end{aligned}
\end{equation}
This expression allows us to find some $b^k_J\in \B_J $ for $J\in \mathcal{H}(w)$ such that
\begin{equation}
    v_k=\sum_{J \in \mathcal{I}}D_J\left( \frac{D^*_J(x_0+t_kw_k)}{\|D^*_J(x_0+t_kw_k)\|}\right)+\sum_{J \in \mathcal{K}(w)}D_J\left( \frac{D^*_Jw_k}{\|D^*_Jw_k\|}\right)+\sum_{J \in \mathcal{H}(w)}D_J(b^k_J).
\end{equation}
For any $J\in \mathcal{I}$, by linear approximation, we obtain from 
\eqref{def:AJ} that 
\begin{equation}\label{eq:Hes2}
\dfrac{D_J(D^*_{J}(x_0+t_kw_k))}{\|D^*_{J}(x_0+t_kw_k)\|}=D_J(e_J+t_k A_JD^*_{J}w_k+o(t_k)).
\end{equation}  
 Since $v_{k}$ belongs to $\Im \Phi^{*}$,  it follows from the definition of $A_\mathcal{I}$ in \eqref{def:AI} that
\begin{equation}\label{eq:vkB}
\begin{aligned}
    D_\mathcal{I}e_\mathcal{I}+t_k D_\mathcal{I}A_\mathcal{I}D^*_\mathcal{I}w_k+o(t_k)&=v_k - \sum_{J \in \mathcal{K}(w)}D_J\left( \frac{D^*_Jw_k}{\|D^*_Jw_k\|}\right)+\sum_{J \in \mathcal{H}(w)}D_J(b^k_J)\\ &\in \Im\Phi^*+
D_{\mathcal{I}^c}(\mathbb{S}_{\mathcal{K}(w)} \times \B_{\mathcal{H}(w)}). 
\end{aligned}
\end{equation}
 Moreover, it follows from \eqref{eq:Hes2} that $D_\mathcal{I}e_\mathcal{I}+t_k D_\mathcal{I}A_\mathcal{I}D^*_\mathcal{I}w_k+o(t_k)\in D_\mathcal{I}(\mathbb{S}_\mathcal{I})$. By combining this with \eqref{eq:vkB}, we deduce from the definition of the tangent cone \eqref{eq:Tangent} that 
\[D_\mathcal{I}A_{\mathcal{I}}D^*_{\mathcal{I}}w  \in  T_{ \left({\rm Im}\Phi^*+D_{\mathcal{I}^c}\left(\mathbb{S}_{\mathcal{K}(w)}\times\mathbb{B}_{\mathcal{H}(w)}\right)\right)\cap D_{\mathcal{I}}(\mathbb{S}_{\mathcal{I}})}(D_{\mathcal{I}}e_\mathcal{I}),
\]
which verifies the inclusion \eqref{eq:Tg}. By Theorem~\ref{thm:Rob}, \eqref{con:full} implies stable recovery at $x_0$.

Next, let us justify the inclusion ``$\supset$'' in \eqref{eq:Tg} when $D^*$ is surjective. Pick any $w\in W(x_0)$ in  the right-hand side of \eqref{eq:Tg}. We find sequences $t_k \downarrow 0$ and $ z^k \to D_{\mathcal{I}} A_\mathcal{I} D^*_{\mathcal{I}}w $ such that 
\begin{equation}\label{eq:zk}
D_{\mathcal{I}}e_{\mathcal{I}} + t_kz^k \in \left (\Im\Phi^*+D_{\mathcal{I}^c}(\mathbb{S}_{\mathcal{K}(w)} \times \B_{\mathcal{H}(w)})\right) \cap D_{\mathcal{I}}(\mathbb{S}_\mathcal{I}).
\end{equation}
To ensure that $w\in \Ker \Phi\cap T_{\partial R^*({\rm Im}\, \Phi^*)}(x_0)$, we will construct sequences $\{w_k\}\subset \R^n$ and $\{v^k\}\subset \Im \Phi^*$ such that  $v^k\in \partial R(x_0+t_kw^k)$ and $w_k\to w$. Since $D^*$ is surjective, $D_\mathcal{I}$ is injective. As $z^k\in \Im D_\mathcal{I}$ by \eqref{eq:zk}, there exists a unique $u^k$ such that $z^k = D_\mathcal{I}u^k$. It follows that $u^k\to A_\mathcal{I} D^*_{\mathcal{I}}w$. This together with \eqref{eq:zk} gives us that $e_{\mathcal{I}}+t_ku^k$ belongs to $\mathbb{S}_\mathcal{I}$ due to the injectivity of  $D_\mathcal{I}$. Define
\begin{equation}\label{eq:wkI}
y^k_J:= \la e_J, D_J^*w\ra e_J + (\|\oy_J\| + t_k \la e_{J}, D^*_Jw \ra )u^k_J\quad \mbox{for any}\quad J \in \mathcal{I}.
\end{equation}
Note that 
\[\begin{aligned}
\oy_J+t_ky_J^k&=\|\oy_J\|e_J + t_k\la e_J, D^*_Jw \ra e_J +t_k(\|\oy_J\|  + t_k \la e_{J}, D^*_Jw \ra )u_J^k\\
&=(\|\oy_J\|+t_k\la e_J, D^*_Jw\ra)(e_J+t_ku^k_{J}).
\end{aligned}
\]
Since $e_J+t_ku^k_{J} \in \mathbb{S}_J$, we have $\|\oy_J+t_ky^k_J\|=\|\oy_J\|+t_k \la e_J, D^*_{J}w \ra$,  which is positive for sufficiently large $k$, as $\|\oy_J\|>0$ for $J\in \mathcal{I}$. The above equality implies that  
\begin{equation}\label{eq:vkI}
D_\mathcal{I}e_\mathcal{I}+t_kz^k=D_\mathcal{I}(e_\mathcal{I}+t_ku^k)=D_\mathcal{I}\left(\disp\prod_{J\in \mathcal{I}}\dfrac{\oy_J+t_ky^k_J}{\|\oy_J+t_ky^k_J\|}\right).
\end{equation}
As $u^k\to A_\mathcal{I} D^*_{\mathcal{I}}w$, it follows  from \eqref{eq:wkI}, \eqref{eq:Hes1}, and \eqref{def:AI} that 
\begin{equation}\label{eq:limwkj}
\begin{aligned}
\lim\limits_{k \to \infty} y^k_J &=  \la e_J, D_J^*w \ra e_J + \|\oy_J\| A_{J}D^*_{J}w \\
&=  \la e_J, D_J^*w \ra e_J+ \|\oy_J\|\left ( \frac{1}{\|\oy_J\|}\Id_J -\frac{1}{\|\oy_J\|^3}\oy_J(\oy_J)^T\right)D^*_{J}w  \\
&=  \la e_J, D_J^*w\ra e_J+ \left( \Id_J -\frac{\oy_J}{\|\oy_J\|}\frac{(\oy_J)^T}{\|\oy_J\|}\right) D^*_{J}w\\
&=  \la e_J, D_J^*w \ra e_J+ D^*_{J}w - e_J \la e_J, D_J^*w \ra \\
&= D^*_{J}w.
\end{aligned}
\end{equation}
Due to \eqref{eq:zk}, there exists some $b^k\in \mathbb{S}_{\mathcal{K}(w)} \times \mathbb{B}_{\mathcal{H}(w)}$ such that 
\begin{equation}\label{eq:uk2}
D_{\mathcal{I}}(e_\mathcal{I}+t_ku^k)+D_{\mathcal{I}^c}(b^k) \in \Im\Phi^*.
\end{equation}
Since $b^k\in \mathbb{S}_{\mathcal{K}(w)} \times \mathbb{B}_{\mathcal{H}(w)}$, we may suppose that $b^k \to b\in\mathbb{S}_{\mathcal{K}(w)} \times \mathbb{B}_{\mathcal{H}(w)}$. By \eqref{eq:Sub}, we deduce that $D(e_\mathcal{I},b)\in \Delta(x_0)$. As $w\in W(x_0)$, for any $J\in \mathcal{K}(w)$, it follows from  Proposition \ref{prop:Crit} that  $D^*_Jw=\lm_Jb_J$ for some  $\lm_J\ge 0$, which implies that $\lm_J=\|D^*_Jw\|$ and $\|D^*_Jw\|b^k_J \to D^*_Jw $.
Let us set  
\begin{equation}\label{eq:wKH}
       y^k:=\left\{\begin{array}{ll} \la e_J, D_J^*w \ra e_J +(\|\oy_J\| + t_k \la e_{J}, D^*_Jw \ra )u^k_J\quad &\mbox{for}\quad J\in \mathcal{I}\;\; \mbox{  from  }\;\;\eqref{eq:wkI}\\ 
   \|D^*_Jw\| b^k_J \quad &\mbox{for}\quad J\in \mathcal{K}(w)\\ 
   0& \mbox{for}\quad J\in \mathcal{H}(w).
\end{array}\right.
\end{equation}
As $D^*_Jw=0$ for any $w\in \mathcal{H}(w)$ by \eqref{eq:KHw}, we obtain from \eqref{eq:limwkj} and \eqref{eq:wKH} that $y^k\to D^*w$.

Next, let us build up the sequence $\{w_k\}$ satisfying $D^*w_k=y^k$ with $w_k\to w$, as $k\to \infty$. Define  $a:=w-(D^*)^{\dag}D^*w$, where $(D^*)^{\dag}$ is the {\em Moore-Penrose inverse} of $D^*$. We have $a\in \Ker D^*$. Set $w_k:=(D^*)^{\dag}y^k+a$. As $y^k\to D^*w$, it follows that  
\[
w_k-w=(D^*)^\dag(y^k-D^*w)\to 0. 
\]
Since $a\in \Ker D^*$ and $D^*$ is surjective, $D^*w_k=y^k$. Finally, note from \eqref{eq:vkI}, \eqref{eq:uk2},  \eqref{eq:wKH}, and the fact that $\|b^k_J\|=1$ for any $J\in \mathcal{K}(w)$ that 
\[
v_k:=D(e_\mathcal{I}+t_ku^k,b^k)=D\left(\prod_{J\in \mathcal{I}}\dfrac{\oy_J+t_ky^k_J}{\|\oy_J+t_ky^k_J\|}, \prod_{J\in \mathcal{K}(w)}\dfrac{D^*_{J}w_k}{\|D^*_{J}w_k\|},b^k_{\mathcal{H}(w)}\right)\in D\partial \|\oy+t_kD^*w_k\|_{1,2}\cap\Im\Phi^*.
\]
It follows that  $v_k\in \partial R(x_0+t_kw_k)\cap \Im\Phi^*$, which implies $x_0+t_kw_k\in \partial R^*(\Im \Phi^*)$. As $w_k\to w$, we have $w\in 
 T_{\partial R^*({\rm Im}\, \Phi^*)}(x_0)$. The inclusion ``$\supset$'' in \eqref{eq:Tg} is verified. In this case, condition~\eqref{con:full} is equivalent to stable recovery at $x_0$ thanks to Theorem~\ref{thm:Rob}. 
 \end{proof}

The 2D discrete gradient operator \eqref{eq:TV} is never surjective. It is not clear to us if the inclusion in \eqref{eq:Tg} can turn to equality in this case. Nonetheless, for the group sparsity regularized linear inverse problems, $D^*$ is the identity matrix, which is surjective obviously. The following result gives a simplified form of Theorem~\ref{thm:Full} for this case.

\begin{Corollary}[Stable recovery for the group sparsity regularized linear inverse problems]\label{thm:aa}
Let  $R=\|\cdot\|_{1,2}$ be the $\ell_1/\ell_2$ norm defined in \eqref{eq:l12} and $x_0$ be a minimizer of problem \eqref{p:P12} with $D$ being the identity matrix. Then we have 
\begin{eqnarray}\label{eq:Tg1}
\Ker \Phi \cap T_{\partial R^*({\rm Im}\, \Phi^*)}(x_0) = \left\{w\in W(x_0)|\; A_\mathcal{I}w_\mathcal{I}\in  T_{ \Phi^*_\mathcal{I}(\Phi^{*}_{\mathcal{I}^{c}})^{-1}{(\mathbb{S}_{\mathcal{K}(w)} \times \mathbb{B}_{\mathcal{H}(w)}}) \cap \mathbb{S}_\mathcal{I}}(e_\mathcal{I})\right\}.
\end{eqnarray}
Consequently, $x_0$ is stably recoverable if only if  
\begin{equation}\label{eq:Tg2}
    \left\{w\in W(x_0)|\; A_\mathcal{I}w_\mathcal{I}\in  T_{ \Phi^*_\mathcal{I}(\Phi^{*}_{\mathcal{I}^{c}})^{-1}{(\mathbb{S}_{\mathcal{K}(w)} \times \mathbb{B}_{\mathcal{H}(w)}}) \cap \mathbb{S}_\mathcal{I}}(e_\mathcal{I})\right\}=\{0\}.
\end{equation}
\end{Corollary}
\begin{proof} In the spirit of Theorem~\ref{thm:Full}, we just need to show that  the right-hand side of \eqref{eq:Tg} is exactly the right-hand side of \eqref{eq:Tg1}.
Let $v\in \Delta(x_0)$ be a dual certificate of problem \eqref{p:P12} with $D=\Id$. Note that the set $W(x_0)$ in \eqref{eq:KCrit} reduces to 
\begin{equation*}\label{eq:KCrit2}
W(x_0)=\{w\in \Ker\Phi|\; w_J\in \R_+v_J, J\in \mathcal{K}, w_J=0, J\in \mathcal{H}\}.
\end{equation*}
The condition inside the right-hand side of \eqref{eq:Tg} is simplified by
\begin{equation}\label{eq:AwI}
(A_{\mathcal{I}}w_\mathcal{I},0_{\mathcal{I}^c}) \in T_{ \left({\rm Im}\,\Phi^*+\{0_{\mathcal{I}}\} \times \mathbb{S}_{\mathcal{K}(w)}\times\mathbb{B}_{\mathcal{H}(w)}\right)\cap \left(\mathbb{S}_{\mathcal{I}}\times \{0_{\mathcal{I}^c}\}\right)}\left(e_{\mathcal{I}}\times \{0_{\mathcal{I}^c}\}\right).
\end{equation}
Note that the set in the above tangent cone can be formulated as   \begin{equation*}
\left\{(\Phi^*_{\mathcal{I}}u,0_{\mathcal{I}^c})|\; \Phi^*_{\mathcal{I}}u\in  \mathbb{S}_{\mathcal{I}}, \Phi^*_{\mathcal{I}^c}u \in \mathbb{S}_{\mathcal{K}(w)}\times \mathbb{B}_{\mathcal{H}(w)} \right\}=\left(\Phi^*_\mathcal{I}(\Phi^{*}_{\mathcal{I}^{c}})^{-1}{(\mathbb{S}_{\mathcal{K}(w)} \times \mathbb{B}_{\mathcal{H}(w)}}) \cap \mathbb{S}_\mathcal{I}\right)\times \{0_{\mathcal{I}^c}\}.
\end{equation*}
Thus, \eqref{eq:AwI} is equivalent to
\[
A_{\mathcal{I}}w_\mathcal{I} \in T_{\Phi^*_\mathcal{I}(\Phi^{*}_{\mathcal{I}^{c}})^{-1}{(\mathbb{S}_{\mathcal{K}(w)} \times \mathbb{B}_{\mathcal{H}(w)}}) \cap \mathbb{S}_\mathcal{I}}(e_\mathcal{I}).
\]
This verifies \eqref{eq:Tg1} and completes the proof. 
\end{proof}


A drawback of  expressions \eqref{eq:Tg} and \eqref{eq:Tg1} is that the index sets $\mathcal{K}(w)$ and $\mathcal{H}(w)$ depend on $w$. Given  $v\in \partial\|\oy\|_{1,2}$ with $Dv\in \Delta(x_0)$, the index sets  $\mathcal{K}$ and $\mathcal{H}$ defined in \eqref{def:KH} are independent of $w$. Note from \eqref{eq:KCrit} that 
\begin{equation}\label{eq:KKw}  
\mathcal{K}\supset \mathcal{K}(w)\quad \mbox{and}\quad \mathcal{H}\subset \mathcal{H}(w).
\end{equation}
Later on, we will come up with a sufficient condition for stable recovery by using $\mathcal{K}$ and $\mathcal{H}$ only. One condition that can make $\mathcal{K}(w)$ fixed is $\mathcal{K}=\emptyset$. This condition is known as the {\em Nondegeneracy Source Condition} \cite{GSH11,FPVDS13,FNT23,VPF15} at $x_0$ for problem~\eqref{p:P12}:
\begin{equation}\label{def:ND}
\exists\, v\in \Delta(x_0): \; v\in \ri (\partial R(x_0)),
\end{equation}
where $\ri (\partial R(x_0))$ is the relative interior of $\partial R(x_0)$ from \eqref{def:Ri}. 
As $\ri D\left(\partial \|\oy\|_{1,2}\right)=D\left(\ri\partial \|\oy\|_{1,2}\right)$ by \cite[Theorem~6.6]{R70}, we have  \begin{equation}\label{eq:ri}
\ri (\partial R(x_0))=\ri D\left(\partial \|\oy\|_{1,2}\right)=D\left\{z\in \R^p|\; z_\mathcal{I}=e_\mathcal{I}, \|z_J\|<1\, \forall\, J\in \mathcal{I}^c\right\}.
\end{equation}
The Nondegeneracy Source Condition \eqref{def:ND}
means that there exists $v\in \ri\partial \|\oy\|_{1,2}$ with $Dv\in \Delta(x_0)$ such that the corresponding index set $\mathcal{K}$ in \eqref{def:KH} is empty. According to \eqref{eq:KKw}, $\mathcal{K}(w)=\emptyset$ for any $w\in W(x_0)$.   It is known from \cite[Theorem~3.10 and Theorem~4.6]{FNT23} that Nondegeneracy Source Condition together with the so-called {\em Restricted Injectivity} \cite{FPVDS13,CR13} (different from \eqref{con:G11})
 \begin{equation}\label{def:RI}
     \Ker \Phi\cap \Ker D^*_{\mathcal{I}^c}=\{0\}
 \end{equation}
 is equivalent to sharp minimum at $x_0$ for problem~\eqref{p:P12}, which guarantees stable recovery at $x_0$. We show next that the above Restricted Injectivity is not necessary for stable recovery.

\begin{Corollary}[Stable recovery and solution uniqueness under the Nondegeneracy Source Condition]\label{cor:ND} Let $x_0$ be an optimal solution of problem~\eqref{p:P12}. Suppose that the Nondegeneracy Source Condition~\eqref{def:ND} is satisfied. Then   $x_0$ is stably recoverable  if and only if it is a unique (strong) minimizer of problem~\eqref{p:P12}. 
\end{Corollary}
\begin{proof} As solution uniqueness is necessary for stable recovery, it suffices to show that stable recovery occurs at $x_0$ if $x_0$ is a unique (strong) minimizer of problem~\eqref{p:P12} when the Nondegeneracy Source Condition~\eqref{def:ND} is satisfied. For any $w\in \Ker \Phi\cap T_{\partial R^*({\rm Im}\, \Phi^*)}(x_0)$, we get from \eqref{eq:Tg} that $w\in W(x_0)$ and $D_\mathcal{I}A_\mathcal{I}D^*_\mathcal{I}w\in \Im \Phi^*+\Im D_{\mathcal{I}^c}$. As $\mathcal{K}=\emptyset$, Proposition~\ref{prop:Crit} tells us that $D^*_\mathcal{H}w=D^*_{\mathcal{I}^c}w=0$, i.e., $w\in \Ker \Phi \cap \Ker D^*_{\mathcal{I}^c}$. It follows that 
\begin{equation}\label{eq:Zero}
0=\la D_\mathcal{I}A_\mathcal{I}D^*_\mathcal{I}w,w\ra=\la A_\mathcal{I}D^*_\mathcal{I}w,D^*_\mathcal{I}w\ra=\sum_{J\in \mathcal{I}}\la A_JD^*_Jw,D^*_Jw\ra,
\end{equation}
which implies that $D_J^*w\in \Ker A_J=\R\{\oy_J\}$ due to \eqref{def:AJ}. Hence, we have $w\in \mathcal{E}\cap W(x_0)$ by \eqref{def:ME}. Since $x_0$ is a unique (strong) solution of problem~\eqref{p:P12},  Theorem~\ref{thm:Equi} tells us that $w=0$. This ensures that $\Ker \Phi \cap T_{\partial R^*({\rm Im} \Phi^*)}(x_0)=\{0\}$, which means  stable recovery occurs at $x_0$ by Theorem~\ref{thm:Rob}. 
\end{proof}

\begin{Remark}[\rm Stable recovery without Restricted Injectivity]\label{F2}
\rm Recalling Example \ref{ex:l12} at which we show that stable recovery occurs at $x_0$ without sharp minima. Note that the dual certificate in that example  $v=\left(\frac{1}{\sqrt{2}},\frac{1}{\sqrt{2}},0\right)$ belongs to $\ri (\partial R(x_0))$, i.e., the Nondegeneracy Source Condition is satisfied. However, note that 
\[
\{w\in \Ker \Phi|\, w_\mathcal{H}=0\}=\Ker \Phi \cap \Ker \Id^{*}_{\mathcal{I}^c} = \R(1,-1,0) \cap (\R \times \R \times \{0\})=\R(1,-1,0)  \ne \{0\}, 
\]
which means both the Restricted Injectivity in \eqref{def:RI} and \eqref{con:G11} fail.
$\hfill\triangle$
\end{Remark}

We need to understand more about tangent cone on the right-hand side of \eqref{eq:Tg}. Observe that 
\begin{equation}\label{eq:TgSub1}
T_{ \left({\rm Im}\,\Phi^*+D_{\mathcal{I}^c}\left(\mathbb{S}_{\mathcal{K}(w)}\times\mathbb{B}_{\mathcal{H}(w)}\right)\right)\cap D_{\mathcal{I}}(\mathbb{S}_{\mathcal{I}})}(D_{\mathcal{I}}e_\mathcal{I})\subset T_{ \left({\rm Im}\,\Phi^*+D_{\mathcal{I}^c}\left(\mathbb{S}_{\mathcal{K}(w)}\times\mathbb{B}_{\mathcal{H}(w)}\right)\right)}(D_{\mathcal{I}}e_\mathcal{I}). 
\end{equation}
For any $v\in \partial \|\oy\|_{1,2}$ with $Dv\in \Delta(x_0)$ and the index pair $\mathcal{K}$ and $\mathcal{H}$ in \eqref{def:KH}, we derive from \eqref{eq:KKw} that  
\begin{equation}\label{eq:TgSub2}
T_{ \left({\rm Im}\,\Phi^*+D_{\mathcal{I}^c}\left(\mathbb{S}_{\mathcal{K}(w)}\times\mathbb{B}_{\mathcal{H}(w)}\right)\right)}(D_{\mathcal{I}}e_\mathcal{I})\subset T_{ \left({\rm Im}\,\Phi^*+{\rm Im}\, D_\mathcal{H}+D_{\mathcal{K}}\left(\mathbb{S}_{\mathcal{K}(w)}\times\mathbb{B}_{\mathcal{K}\setminus\mathcal{K}(w)}\right)\right)}(D_{\mathcal{I}}e_\mathcal{I}). 
\end{equation}
Note from \eqref{eq:Sub} that $D_{\mathcal{I}}e_\mathcal{I}=Dv-D_{\mathcal{H}}v_\mathcal{H}-D_{\mathcal{K}}v_\mathcal{K}\in\Im \Phi^*+{\rm Im}\, D_\mathcal{H}-D_{\mathcal{K}}v_\mathcal{K}$. It follows from the formula of tangents under set addition in \cite[Exercise 6.44]{RW98} and the tangent of image set \cite[Theorem~6.43]{RW98} that 
\begin{equation*}
\begin{aligned}
T_{ \left({\rm Im}\,\Phi^*+{\rm Im}\, D_\mathcal{H}+D_{\mathcal{K}}\left(\mathbb{S}_{\mathcal{K}(w)}\times\mathbb{B}_{\mathcal{K}\setminus\mathcal{K}(w)}\right)\right)}(D_{\mathcal{I}}e_\mathcal{I})&\supset \Im \Phi^*+\Im D_\mathcal{H}+T_{\left(D_{\mathcal{K}}\left(\mathbb{S}_{\mathcal{K}(w)}\times\mathbb{B}_{\mathcal{K}\setminus\mathcal{K}(w)}\right)\right)}(-D_{\mathcal{K}}v_\mathcal{K})\\
&\supset \Im \Phi^*+\Im D_\mathcal{H}+D_{\mathcal{K}}T_{\left(\mathbb{S}_{\mathcal{K}(w)}\times\mathbb{B}_{\mathcal{K}\setminus\mathcal{K}(w)}\right)}(-v_\mathcal{K}).
\end{aligned}
\end{equation*}
Of course, the above ``$\supset$'' inclusions are not helpful for obtaining an upper estimate of the left-hand side set in \eqref{eq:TgSub1}. But if they become equalities, we show next that stable recovery is equivalent to solution uniqueness.

\begin{Theorem}[Stable recovery and solution uniqueness]\label{thm:Equiva} Let  $x_0$ be an optimal solution of problem \eqref{p:P12} and $v \in \partial \|D^{*}x_0\|_{1,2}$ with $Dv\in \Delta(x_0)$. If $w\in \Ker\Phi \cap T_{\partial R^*({\rm Im}\, \Phi^*)}(x_0)$ satisfies the following condition 
\begin{eqnarray}\label{con:Tan}
 T_{\left({\rm Im}\, \Phi^* + {\rm Im} \,D_{\mathcal{H}} + D_{\mathcal{K}}(\mathbb{S}_{\mathcal{K}(w)} \times \B_{\mathcal{K}\setminus \mathcal{K}(w)})\right)}(D_{\mathcal{I}}e_{\mathcal{I}})={\rm Im}\, \Phi^* + {\rm Im}\, D_{\mathcal{H}} + D_{\mathcal{K}}T_{\mathbb{S}_{\mathcal{K}(w)} \times \B_{\mathcal{K}\setminus \mathcal{K}(w)}}(-v_{\mathcal{K}}), 
\end{eqnarray}
we have $w\in \mathcal{E}$. Consequently, if condition~\eqref{con:Tan} holds for any $w \in W(x_0) $, $x_0$ is stably recoverable if and only if it is the unique minimizer of problem~\eqref{p:P12}. 
\end{Theorem}
\begin{proof} Let us start to prove  the first part  by picking any $w\in \Ker\Phi \cap T_{\partial R^*({\rm Im}\, \Phi^*)}(x_0)$ satisfying condition~\eqref{con:Tan}.  By \eqref{eq:Tg}, we have 
\[
D_\mathcal{I}A_{\mathcal{I}}D^*_{\mathcal{I}}w  \in  T_{ \left({\rm Im}\Phi^*+D_{\mathcal{I}^c}\left(\mathbb{S}_{\mathcal{K}(w)}\times\mathbb{B}_{\mathcal{H}(w)}\right)\right)\cap D(\mathbb{S}_{\mathcal{I}})}(D_{\mathcal{I}}e_\mathcal{I}) 
.
\]
It follows from  \eqref{eq:TgSub1}, \eqref{eq:TgSub2}, and \eqref{con:Tan} 
 that
\begin{equation*}\label{eq:SKL}
D_\mathcal{I}A_{\mathcal{I}}D^*_{\mathcal{I}}w\in \Im \Phi^* + \Im D_{H} + D_{\mathcal{K}}T_{\mathbb{S}_{\mathcal{K}(w)} \times \B_{\mathcal{K}\setminus \mathcal{K}(w)}}(-v_\mathcal{K}).
\end{equation*}
Note also that 
\begin{equation}\label{eq:TSK}
T_{\mathbb{S}_{\mathcal{K}(w)}\times \B_{\mathcal{K}\setminus\mathcal{K}(w)}}(-v_{\mathcal{K}})=\{z\in \R^{\mathcal{K}}|\; \la z_J,v_J\ra=0, J\in \mathcal{K}(w), \la z_J,v_J\ra\ge0, J\in \mathcal{K}\setminus\mathcal{K}(w)\}. 
\end{equation}
Hence there exist $y \in \Im \Phi^* + \Im D_{\mathcal{H}}$ 
and $z \in T_{{\mathbb{S}_{\mathcal{K}(w)}} \times \B_{\mathcal{K}\setminus \mathcal{K}(w)}}(-v_\mathcal{K})$ such that $D_\mathcal{I}A_{\mathcal{I}}D^*_{\mathcal{I}}w = y + D_\mathcal{K}z$. 
As $w\in W(x_0)$ by \eqref{eq:Tg}, we get from Proposition~\ref{prop:Crit} that  $w\in \Ker \Phi\cap \Ker D_\mathcal{H}^*$, which gives us that $\la y,w\ra=0$. Consequently, we obtain that 
\[
\begin{aligned}
 \la D_\mathcal{I}A_{\mathcal{I}}D^*_{\mathcal{I}}w,w\ra &= \la y + D_{\mathcal{K}}z,w \ra \\
 &=\la z,D^*_\mathcal{K}w\ra\\
 &= \la z_{\mathcal{K}(w)}, D^{*}_{\mathcal{K}(w)}w \ra + \la z_{\mathcal{K}\setminus \mathcal{K}(w)}, D^{*}_{\mathcal{K}\setminus \mathcal{K}(w)}w \ra\\
&=\sum_{J\in \mathcal{K} \setminus \mathcal{K}(w)} \la z_{J}, D^{*}_{J}w\ra \qquad (\text{by}\; \eqref{eq:KCrit},\; \eqref{eq:KKw}, \;\text{and}\; \eqref{eq:TSK})\\
&=  0. \qquad (\text{as}\; \mathcal{K}\setminus\mathcal{K}(w)\subset \mathcal{H}(w))
\end{aligned}
\]
It is similar to \eqref{eq:Zero} that  $D^*_Jw\in \Ker A_{J}=\R\{\oy_J\}$ for any $J\in \mathcal{I}$, which means $w\in \mathcal{E}$ by \eqref{def:ME}. This verifies the first part.

For the second part, let us suppose that condition~\eqref{con:Tan} holds for any $w\in W(x_0)$ and $x_0$ is the unique solution. It follows from Theorem~\ref{thm:Equi} that $
W(x_0)\cap  \mathcal{E}=\{0\}.$
 By using the first part, condition~\eqref{con:SOSC} is satisfied. By Theorem~\ref{thm:Rob}, $x_0$ is stably recoverable.  The proof is completed. 
\end{proof}

All we have to do now is to find a condition that \eqref{con:Tan} is satisfied. This is achievable by using the chain rule of tangent cone over a linear image \cite[Proposition~4.3.9]{AF90}.

\begin{Corollary}[A sufficient condition for stable recovery]\label{Cor:Equi} Let $x_0$ be a minimizer of problem \eqref{p:P12} and  $v \in \partial \|D^{*}x_0\|_{1,2}$ with $ Dv \in \Delta(x_0)$ and index pair $\mathcal{K}$ and $\mathcal{H}$ in \eqref{def:KH}.  Stable recovery occurs at $x_0$ if and only if it is the unique (strong) minimizer of problem~\eqref{p:P12} provided that the following condition is satisfied
\begin{equation}\label{con:NoTg}
\begin{cases}D_{\mathcal{K}}z \in \Im \Phi^* + \Im D_{\mathcal{H}} \\
\la z_J, v_J \ra = 0, J \in \mathcal{K}(w)\\
\la z_J, v_J \ra \ge 0, J \in \mathcal{K}\setminus\mathcal{K}(w)
\end{cases}
\Longrightarrow\quad  z = 0\quad \mbox{for any}\qquad w\in W(x_0).
\end{equation}
Consequently, stable recovery occurs at $x_0$ if the following two conditions hold  
\begin{equation}\label{con:Now}
W(x_0)\cap \mathcal{E}=\{0\}\quad \text{and}\quad \begin{cases}D_{\mathcal{K}}z \in \Im \Phi^* + \Im D_{\mathcal{H}} \\
\la z_J, v_J \ra \ge 0, J \in \mathcal{K}
\end{cases}
\Longrightarrow\quad  z = 0.
\end{equation}
\end{Corollary}
\begin{proof} For the first part, we just need to show that condition~\eqref{con:NoTg} implies condition \eqref{con:Tan} for any $w\in W(x_0)$. Define the linear operator  $\Lambda:\R^n\times \R^{\mathcal{K}}\to\R^n$ with $\Lm(y,z):=y+ D_{\mathcal{K}}z$ for any $(y,z)\in \R^n\times \R^{\mathcal{K}}$. Define further the set $\Omega:=(\Im \Phi^* + \Im D_{\mathcal{H}}) \times \left(\mathbb{S}_{\mathcal{K}(w)} \times \B_{\mathcal{K}\setminus\mathcal{K}(w)}\right)$ and the point $\hat y:=Dv-D_{\mathcal{H}}v_\mathcal{H}\in \Im \Phi^*+\Im  D_\mathcal{H}$. Note that $D_\mathcal{I}e_\mathcal{I}=\hat y-D_\mathcal{K}v_\mathcal{K}=\Lambda(\hat y, -v_\mathcal{K})$ and that 
\begin{equation}\label{eq:Tgeq}
T_{{\rm Im} \Phi^* + {\rm Im} D_{\mathcal{H}} + D_{\mathcal{K}}(\mathbb{S}_{\mathcal{K}(w)} \times \B_{\mathcal{K}\setminus\mathcal{K}(w)})}(D_{\mathcal{I}}e_{\mathcal{I}})=T_{\Lambda(\Omega)}(\Lambda(\hat y, -v_\mathcal{K})).   
\end{equation}
According to  the chain rule of tangent cone over a linear image \cite[Proposition~4.3.9]{AF90}, we have
\begin{equation}\label{eq:TgIm}
\cl(\Lambda(T_{\Omega}(\hat y,-v_\mathcal{K})))=T_{\Lambda(\Omega)}\Lm(\hat y,-v_\mathcal{K}),
\end{equation}
provided that 
\begin{equation}\label{con:Ker}
    \Ker \Lambda \cap T_{\O}(\Lambda( \hat y,-v_\mathcal{K})) = \{0\}. 
\end{equation}
Let us show that this condition is equivalent to \eqref{con:NoTg}. Indeed, note first that 
\begin{equation}\label{eq:TgO}
\begin{aligned}
T_{\O}(\Lambda( \hat y,-v_\mathcal{K}))&=T_{\left({\rm Im}\,\Phi^*+{\rm Im}\,D_\mathcal{H}\right)}(\hat y)+T_{\left(\mathbb{S}_{\mathcal{K}(w)} \times \B_{\mathcal{K}\setminus\mathcal{K}(w)}\right)}(-v_\mathcal{K})\\
&={\rm Im}\,\Phi^*+{\rm Im}\,D_\mathcal{H}+T_{\left(\mathbb{S}_{\mathcal{K}(w)} \times \B_{\mathcal{K}\setminus\mathcal{K}(w)}\right)}(-v_\mathcal{K}).
\end{aligned}
\end{equation}
By \eqref{eq:TSK},  $(y,z)\in\Ker \Lambda \cap T_{\O}(\Lambda( \hat y,-v_\mathcal{K}))$ iff $y\in \Im \Phi^*+\Im D_\mathcal{H}$, $z\in T_{\left(\mathbb{S}_{\mathcal{K}(w)} \times \B_{\mathcal{K}\setminus\mathcal{K}(w)}\right)}(-v_\mathcal{K})$, and $y+D_\mathcal{K}z=0$. As $z=0$ implies $y=0$ in the latter equivalence, condition~\eqref{con:Ker} means that 
\[
[D_\mathcal{K}z\in \Im \Phi^*+\Im D_\mathcal{H}\quad \mbox{and}\quad z\in T_{\left(\mathbb{S}_{\mathcal{K}(w)} \times \B_{\mathcal{K}\setminus\mathcal{K}(w)}\right)}(-v_\mathcal{K})]\quad\Longrightarrow\quad z=0,
\]
which is equivalent to \eqref{con:Now} due to the formula of the  tangent cone in \eqref{eq:TSK}. 
Back to the expression \eqref{eq:TgIm} under condition~\eqref{con:Ker}, the closure there is superfluous, as the tangent $T_{\O}(\Lambda( \hat y,-v_\mathcal{K}))$ is a polyhedron due to the formulas \eqref{eq:TgO} and \eqref{eq:TSK}. This together with \eqref{eq:Tgeq} verifies the condition~\eqref{con:Tan} whenever the condition~\eqref{con:NoTg} is satisfied. The first part then follows from Theorem~\ref{thm:Equiva}. 

For the second part, note from \eqref{eq:KKw} that $\mathcal{K}(w)\subset \mathcal{K}$ for any $w\in W(x_0)$. Thus, the condition \eqref{con:Now} is  sufficient for \eqref{con:NoTg}. By the first part and the characterization of solution uniqueness in Theorem~\ref{thm:Equi}, stable recovery occurs at $x_0$ whenever \eqref{con:Now} is satisfied.  
\end{proof}

Next, we give a direct consequence in the case when $D^*$ is the identity matrix. 

\begin{Corollary}[A sufficient condition for stable recovery of group sparsity linear inverse problems]\label{cor:Id}
Let $x_0$ be an optimal solution of problem \eqref{p:P12} when $D^*=\Id$ is the identity matrix. Let $v \in \partial \|x_0\|_{1,2}$ with the index set $\mathcal{K}$ in \eqref{def:KH}.  Stable recovery occurs at $x_0$ if   
\begin{equation}\label{con:NowI}
W(x_0)\cap \mathcal{E}=\{0\}\quad \text{and}\quad \begin{cases}z \in \Phi_\mathcal{K}^*(\Ker \Phi^*_\mathcal{I})\\
\la z_J, v_J \ra \ge 0, J \in \mathcal{K}
\end{cases}
\Longrightarrow\quad  z = 0.
\end{equation}    
\end{Corollary}
\begin{proof}
In the spirit of Corollary \ref{Cor:Equi}, we show that condition \eqref{con:NowI} is deduced  from \eqref{con:Now} when $D^*$ is an identity matrix. The right-hand side condition in \eqref{con:Now} can be reformulated as 
\begin{equation*}
 \begin{cases} (0_{\mathcal{I}},z,0_{\mathcal{H}})\in \Im \Phi^*+(0_{\mathcal{I}},0_{\mathcal{K}},\R^\mathcal{H}) \\
\la z_J, v_J \ra \ge 0, J \in \mathcal{K}
\end{cases}
\Longrightarrow\quad  z = 0,
\end{equation*}    
which is equivalent to
\begin{equation*}
 \begin{cases}\exists\, u\in\R^m:  0_{\mathcal{I}}=\Phi^*_{\mathcal{I}}u, z=\Phi^*_{\mathcal{K}}u\\
\la z_J, v_J \ra \ge 0, J \in \mathcal{K}
\end{cases}
\Longrightarrow\quad  z = 0.
\end{equation*}  
This condition is exactly the one in  \eqref{con:NowI}. The rest of the proof follows from Corollary~\ref{Cor:Equi}.
\end{proof}

Condition~\eqref{con:NoTg} still depends on each $w\in W(x_0)$, but conditions~\eqref{con:Now} and \eqref{con:NowI} fully independent of $w$. It is quite possible to check these two conditions by numerical methods; see our Section~5 for further details. While condition \eqref{con:NoTg} holds if $x_0$ is a sharp minimizer, conditions~\eqref{con:Now} and \eqref{con:NowI} are totally different from sharp minima. Let us consider the following example of group sparsity linear inverse problems with parameters to see when condition~\eqref{con:NowI} is satisfied, but $x_0$ is not a sharp optimal solution, i.e., condition \eqref{con:G11} in \cite{G11,H13} fails.  

\begin{Example}\label{ex:GL}
\rm Consider the following $\ell_1/\ell_2$ optimization problem with four groups
\begin{equation}\label{p:E15}
\min_{x\in \R^6}\quad  R(x) = \sqrt{x_1^2+x_2^2}+ \sqrt{x_3^2+x_4^2}+\sqrt{x_5^2+x_6^2}+ \sqrt{x_7^2+x_8^2}\quad \mbox{subject to}\quad \Phi x=y_0   
\end{equation}
with $\Phi=\begin{pmatrix} 1&0&0&a_4&a_5&a_6&a_7&a_8\\0&1&0&1&0&1&0&1\\0&0&1&b_4&b_5&b_6&b_7&b_8\end{pmatrix} 
$,  $x_0=(0,1,0,0,0,0,0,0)^T$ (a group sparse vector), and $y_0=(0,1,0)^T$. Note that 
 $\mathcal{I}=\{(1,2)\}$ and  $e_\mathcal{I}=(0,1)^T$. It follows from \eqref{def:ME} that 
\begin{equation}\label{eq:exE}
\mathcal{E} = \{x \in \mathbb{R}^8|\; x_1 = 0\}.
\end{equation}
Moreover, observe that  $v=(0,1, 0,1,0, 1,0,1)^T\in \Im \Phi^*\cap \partial R(x_0)$ is a dual certificate of problem~\eqref{p:E15}. Thus $x_0$ is an optimal solution of problem~\eqref{p:E15}. Note also that $\mathcal{K}=\{(3,4), (5,6), (7,8)\}$ and $\mathcal{H}=\emptyset$ by \eqref{def:KH}.  According to  Proposition~\ref{prop:Crit}, $w\in W(x_0)$ if and only if  there exist $s_1,s_2,s_3\ge 0$ such that $w_\mathcal{K}=(0,s_1,0,s_2,0,s_3)^T$ and $w\in \Ker \Phi$, which together with the format of $\Phi$ means 
\begin{equation}\label{eq:Criw}
w=(-a_4s_1-a_6s_2-a_8s_3,-s_1-s_2-s_3,0,s_1,0,s_2,0,s_3)^T \quad \mbox{with}\quad b_4s_1+b_6s_2+b_8s_3=0.
\end{equation}
Combining this with \eqref{eq:exE} gives us that 
\begin{equation}\label{eq:kerbd}
    W(x_0)\cap \mathcal{E}=\{0\}\quad \mbox{if and only if}\quad \Ker \begin{pmatrix} a_4&a_6&a_8\\b_4&b_6&b_8\end{pmatrix}\cap \R^3_+=\{0\},
\end{equation}
which becomes the characterization for solution uniqueness at $x_0$ by Theorem~\ref{thm:Equi}. 
Note also that $\Ker \Phi^*_\mathcal{I}=\R(0,0,1)^T$ and  that 
\[
\Phi^*_\mathcal{K}(\Ker \Phi^*_\mathcal{I})=\R(1,b_4,b_5,b_6,b_7,b_8)^T. 
\]
The right-hand condition in \eqref{con:NowI} means that 
\[ \Phi_\mathcal{K}^*(\Ker \Phi^*_\mathcal{I}) \cap \{z \in \R^6 | \; \la z_J, v_J \ra \ge 0, J \in \mathcal{K} \} = \{0\},
\]
which can be rewritten based on the representation of $\Phi^*_\mathcal{K}(\Ker \Phi^*_\mathcal{I})$ and $v_\mathcal{K} = (0,1,0,1,0,1)$ as
\[
\R(1,b_4,b_5,b_6,b_7,b_8)\cap(\R\times \R_+\times \R\times \R_+\times \R\times \R_+)=\{0\}.
\]
The above condition happens if and only if the system $b_4s \geq 0, b_6s \geq 0,b_8s \geq 0$ for some $s \in \R$ yields $s = 0$, which means $b_4, b_6, b_8$ cannot have the same sign. Thus one of the product $b_4b_6,b_6b_8,b_8b_4$ must be negative, i.e., 
\[
\min\{b_4b_6,b_6b_8,b_8b_4\}<0.
\]
Combining this with \eqref{eq:kerbd} tells us that condition~\eqref{con:NowI} is equivalent to 
\begin{equation}\label{con:NowEx}
   \Ker \begin{pmatrix} a_4&a_6&a_8\\b_4&b_6&b_8\end{pmatrix}\cap \R^3_+=\{0\}\quad \mbox{and}\quad  \min\{b_4b_6,b_6b_8,b_8b_4\}<0. 
\end{equation}
There are many choices of $(a_k,b_k)$, $k=4,\ldots ,8$ satisfying this condition that guarantees stable recovery at $x_0$ as in Corollary~\ref{cor:Id}. At the same time, the sharp minimum at $x_0$ is equivalent to $W(x_0)=\{0\}$, which means by \eqref{eq:Criw} that 
\[
\Ker \begin{pmatrix} b_4&b_6&b_8\end{pmatrix}\cap \R^3_+=\{0\},\;\;\mbox{ i.e., }\;\; b_4,b_6,b_8>0 \;\mbox{or}\; b_4,b_6,b_8<0,
\]
which is completely independent of \eqref{con:NowEx}. For example, by choosing $a_4,a_6,a_8>0$ and $b_4b_6<0$, condition \eqref{con:NowEx} is satisfied. i.e., $x_0$ is stably recoverable, but the above condition and \eqref{con:G11} from \cite{G11,H13} fail, as  $W(x_0)\neq \{0\}$.  This reconfirms that condition~\eqref{con:NowI} is independent of sharp minima. $\hfill\triangle$
\end{Example}

\section{Numerical experiments}
Implementation details: Our experiments are conducted in Python version 3.12.7 on a Macbook Pro with M2 Pro CPU and 16GB of RAM. We solve convex problems using cvxpy package \cite{DB16} with default parameters. All non-convex problems are solved using Gurobi package \cite{G24} version 11.0.3 with ``TimeLimit'' set to $5$ seconds and ``NonConvex'' parameter set to $2$ for Gurobi to handle non-convex objective function using appropriate algorithm. We disable Gurobi's model reduction functionality by setting the ``Presolve'' parameter to $0$, enabling ``Presolve''could cause numerical instability.\\
We mainly show that our sufficient condition in \eqref{con:Now} is quite verifiable with favorable numerical results. Recall problem~\eqref{p:P12} here
\begin{equation}\label{p:check_reco1}
\min_{x \in \R^{n}} \|D^{*}x\|_{1,2}\hspace*{0.2cm} \text{subject to} \hspace*{0.2cm} \Phi x = \Phi x_{0},
\end{equation}
where $D:\R^p\to \R^n$ and $\Phi:\R^n\to \R^m$ are linear operators. For numerical experiments, $\Phi$ is an $\R^{m\times n}$ random matrix drawn from the independent and identically distributed standard normal distribution $\mathcal{N}(0,1)$. This is a convex optimization problem, we use the cvxpy package in \cite{DB16} to solve it and derive an optimal solution $x_{\text{opt}}$. If it satisfies $\dfrac{\|x_{\text{opt}} -  x_{0}\|}{\|x_{0}\|} < 10^{-3}$, the signal $x_{0}$ is said to be recovered exactly. We shall only analyze at those $x_0$. It is quite reasonable to think $x_0$ is the unique solution in this case, but we do have a solution uniqueness test to make sure that.  
Since the main contribution of this paper is about stable recovery without sharp minima, all the cases of sharp minima need to be excluded from our numerical experiments. In order to do so, we compute the so-called {\em Source Coefficient} $\rho(x_0)$ \cite[Remark~4.5 and (6.1)]{FNT23}, which is the optimal value of 
\begin{equation}\label{p:dualcert}
\min_{t>0,\, z\in \R^p,\, w\in \R^m}\qquad t \qquad  \text{subject to}\qquad  D(z+e)+\Phi^*w = 0, \hspace*{0.2cm} \|z_{J}\|^{2} \leq t, J \in \mathcal{I}^c,
\end{equation}
where $e$ is from \eqref{def:e}. According to \cite[Theorem~4.6]{FNT23}, sharp minima occur when $\rho(x_0)<1$, as the Restricted Injectivity \eqref{def:RI} is usually satisfied when choosing random matrix $\Phi$ in \eqref{p:check_reco1}.  This problem can be solved by using the cvxpy package \cite{DB16} again.

Due to the possible error in solving two optimization problems~\eqref{p:check_reco1} and \eqref{p:dualcert}, we classify $x_0$ to be a nonsharp minimizer when  $\rho(x_0) \in (0.95,1.05)$. This is how we obtained Figure~1 with the red curves of non-sharp minimizers.   The optimal solution $\bar z$ achieved from solving \eqref{p:dualcert} gives us a dual pre-certificate $v=\bar z+e\in \partial \|D^*x_0\|_{1,2}$ with $Dv\in \Im \Phi^*$.  With this $v$, we can define the set $\mathcal{K}$ and $\mathcal{H}$ from \eqref{def:KH} as follows
\[
\mathcal{K}:=\{J\in  \mathcal{I}^c|\; \|v_J\| > 0.99\}\quad \mbox{and} \quad\mathcal{H}:= \mathcal{I}^c\setminus\mathcal{K}=\{J\in  \mathcal{I}^c|\; \|v_J\|\le 0.99\}.
\]
By Theorem \ref{thm:Equi}, $x_0$ is the unique optimal solution if
$
W(x_0)\cap \mathcal{E}=\{0\}$. Due to the formula of $W(x_0)$ in \eqref{eq:KCrit},  the latter  is equivalent to the claim that the following nonconvex {\em quadratic programming}  problem
\begin{equation}\label{p:check_unique}
\min_{w \in \R^{n}} \hspace*{0.2cm} -\| w\|^2 \quad 
\text{subject to}\quad 
\Phi w = 0,\;
D^{*}_Jw \in \R \bar{y}_J, J \in \mathcal{I},\; D^{*}_Jw \in \R_{+} v_J, J \in \mathcal{K},\;
D^{*}_{\mathcal{H}}w = 0
\end{equation}
has the optimal value to be zero. We use Gurobi solver \cite{G24} to solve the above problem. When the optimal value is zero, we successfully ensure solution uniqueness and proceed to check stable recovery at recovered solution by \eqref{con:Now}. Similarly,  condition \eqref{con:Now} can be verified  by solving   a non-convex quadratic programming problem, at which Gurobi solver can solve it; see problems \eqref{p:Check1} and \eqref{p:checkTV1} below.

\subsection{Experiment 1}
In the first experiment, when $D^*=\Id$ in \eqref{p:check_reco1}, we carry out numerical tests to see if a unique (non-sharp) optimal solution is stably recoverable by using our sufficient condition \eqref{con:NowI}. Our group-sparse signals $x_0$ are chosen with $2000$ components divided into $100$ non-overlapping groups \cite{RRN12} in three different cases of active (nonzero) groups $5, 10,$ and $15$.  For each number of measurements $m$ (with $10$ units), we generate $100$ random group-sparse vectors $x_0$ together with $100$ random i.i.d. Gaussian  matrices $\Phi\in \R^{m\times 2000}$.  After solving problem~\eqref{p:check_reco1}, we record the percentage of cases of $x_0$ that can be recovered at each $m$ in the green curves in the below figure. 

\begin{figure}[h]
\centering
\includegraphics[width=1\linewidth]{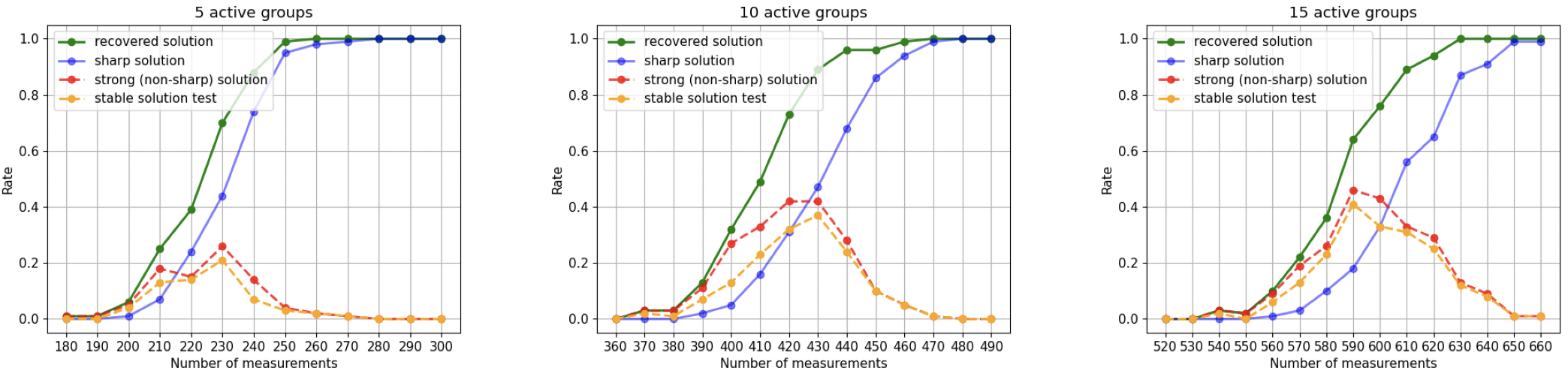}
\caption{Group sparsity problems with different active groups of the signals}
\label{gr_ls}
\end{figure}

Within the recovered solutions from green curve, we use the Source Coefficient from solving problem~\eqref{p:dualcert} to classify sharp solutions ($\rho<0.95$) with  blue curves and nonsharp solutions ($0.95<\rho<1.05$) with  red curves.  All the sharp optimal solutions from the blue curves are stably recoverable  in theory. We only need to check if other solutions on the red curves are stably recoverable by verifying our condition \eqref{con:NowI}. To make sure that the red curves contain unique (strong, nonsharp) solutions, we use Gurobi to solve problem~\eqref{p:check_unique}. The results are highly favorable,  all optimal values of  \eqref{p:check_unique} are super close to $0$ for $100\%$ cases of red curves, i.e., the condition $W(x_0)\cap \mathcal{E} = \{ 0\}$ in \eqref{con:NowI} is satisfied. To verify the other condition in \eqref{con:NowI}, we need the following non-convex minimization problem
\begin{equation}\label{p:Check1}
    \min_{w \in \R^m}\quad -\|\Phi^{*}_{\mathcal{K}}w\|^2\quad 
    \mbox{subject to}\quad
     \Phi^{*}_{\mathcal{I}}w = 0\;\;\mbox{and}\;\;
\la \Phi^*_J w, v_J \ra \ge 0, J \in \mathcal{K}.
\end{equation}
to have the  optimal value to close to $0$ (the possible error is $10^{-6}$). Using Gurobi solver again records cases for the orange curves in the above figures.  

The simulation results are summarized in Figure \ref{gr_ls}, in which we output three line graphs for the three testing scenarios when the number of active groups of $x_0$ are $5,10$ and $15$, respectively. Throughout the three figures, we see a similar pattern that the curves showing the percentages of recovered solutions starts at $0$, meaning no solution is successfully recovered when the number of measurements are low. This curve gradually increases as the number of measurements grows, this is when we also observe the proportion of sharp, strong (non-sharp) and stable solution also increases. Finally, when the number of measurements passes a certain point, the green lines and blue lines both reach the maximum at $1$ and plateau, which implies all $100$ solutions recovered were sharp. At the same time, the red and orange line both drop to $0$, meaning there is no stably recovered strong solution reported.

Furthermore, we see that the gap between the sharp solution (blue line) and the strong solution (red dashed line) increases with the number of active groups when measurements are limited. Our stable solution test (orange line) effectively captures a large portion of the strong non-sharp solutions. Across all three graphs and at every number of measurement tested, we observe that the orange curves nearly match the red ones for the cases of 5 active groups and 15 active groups, implying a majority of strong solutions are stable. This demonstrates the effectiveness of our sufficient condition in Corollary~\ref{cor:Id} when sharp solutions are not achievable.

\subsection{Experiment 2}
In the second experiment, we consider the problem \eqref{p:check_reco1} regularized by the isotropic total variation  seminorm, where $D^*=\nabla$ is the 2D discrete gradient defined in \eqref{eq:TV} and \eqref{eq:total}. In this experiment, multiple images $X_0$ are taken from the Extended MNIST dataset \cite{GSJA17} using the torchvision.datasets package of PyTorch library \cite{A19}. This dataset consists of $145600$ gray-scale images of various digits and letters, each of size $28 \times 28$. Since the dataset is large, we decide to randomly sample six images so that each of them was representative of a data point that belongs to a specific group sparsity distribution of $\nabla x$ in the dataset, we include the sampled images in Figure \ref{random_img}.
\begin{figure}[h]
\centering
\includegraphics[width=0.6\linewidth]{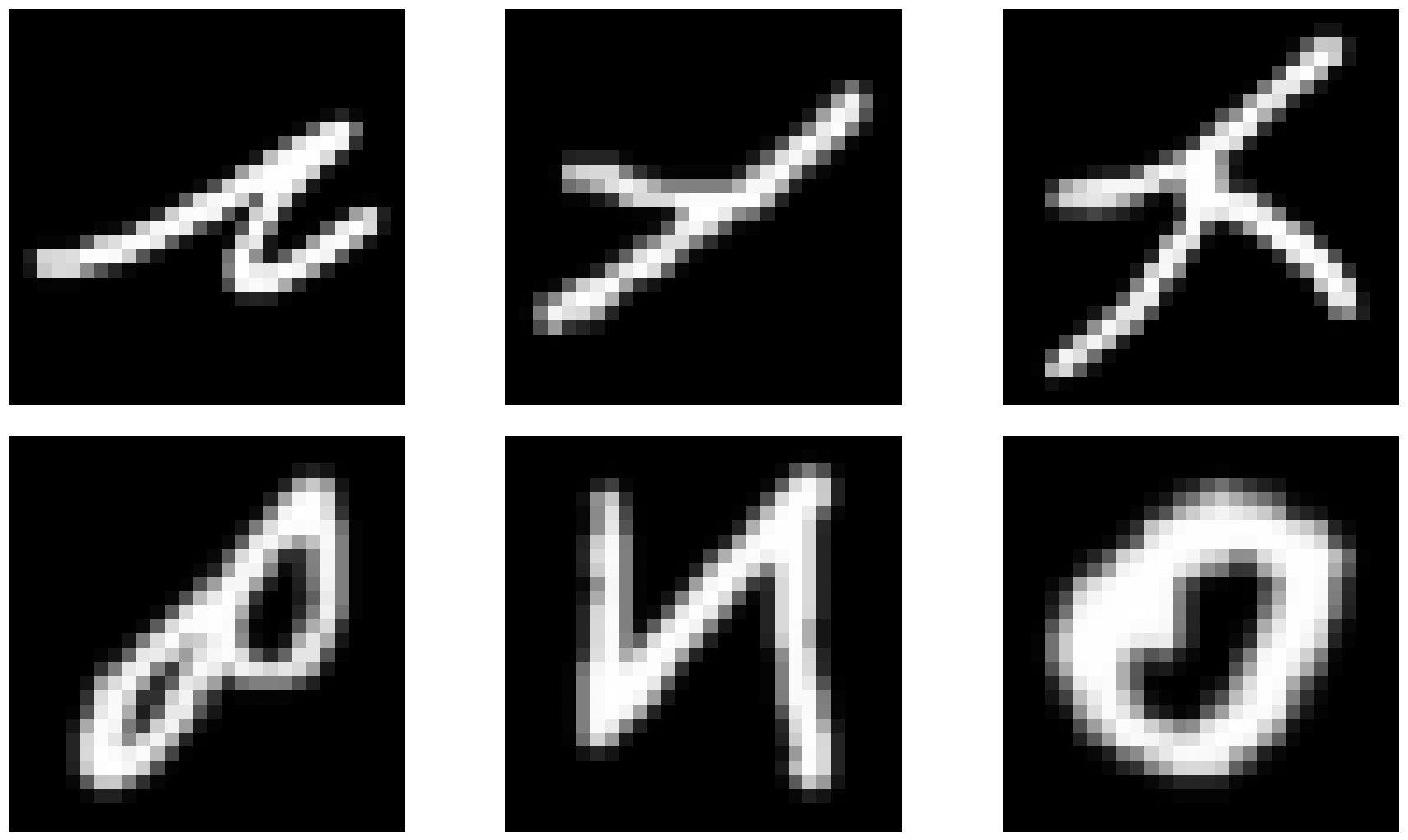} 
\caption{Images randomly sampled from the Extended MNIST dataset}
\label{random_img}
\end{figure}

We vectorize each image as column vectors to mimic all the steps from Experiment 1 with the only difference being the last step of checking our sufficient condition for stable recovery in \eqref{p:Check1}, we use condition \eqref{con:Now} instead of \eqref{con:NowI} and formulate the following non-convex optimization problem for it
\begin{equation}
\label{p:checkTV1}
    \min_{ a,b,u}\quad -\|u\|^2\quad 
    \mbox{subject to}\quad
     \Phi^{*}_{\mathcal{I}}a+D_{\mathcal{H}}b+D_{\mathcal{K}}u = 0\;\;\mbox{and}\;
     \; 
    \la u_J, v_J \ra \ge 0, J \in \mathcal{K}.
\end{equation}
The condition in \eqref{con:Now} is valid if the optimal value of this problem is very close to $0$ again.  We summarize below the results in Figure \ref{low_sparse}. 

\begin{figure}[h]
\centering
\includegraphics[width=1\linewidth]{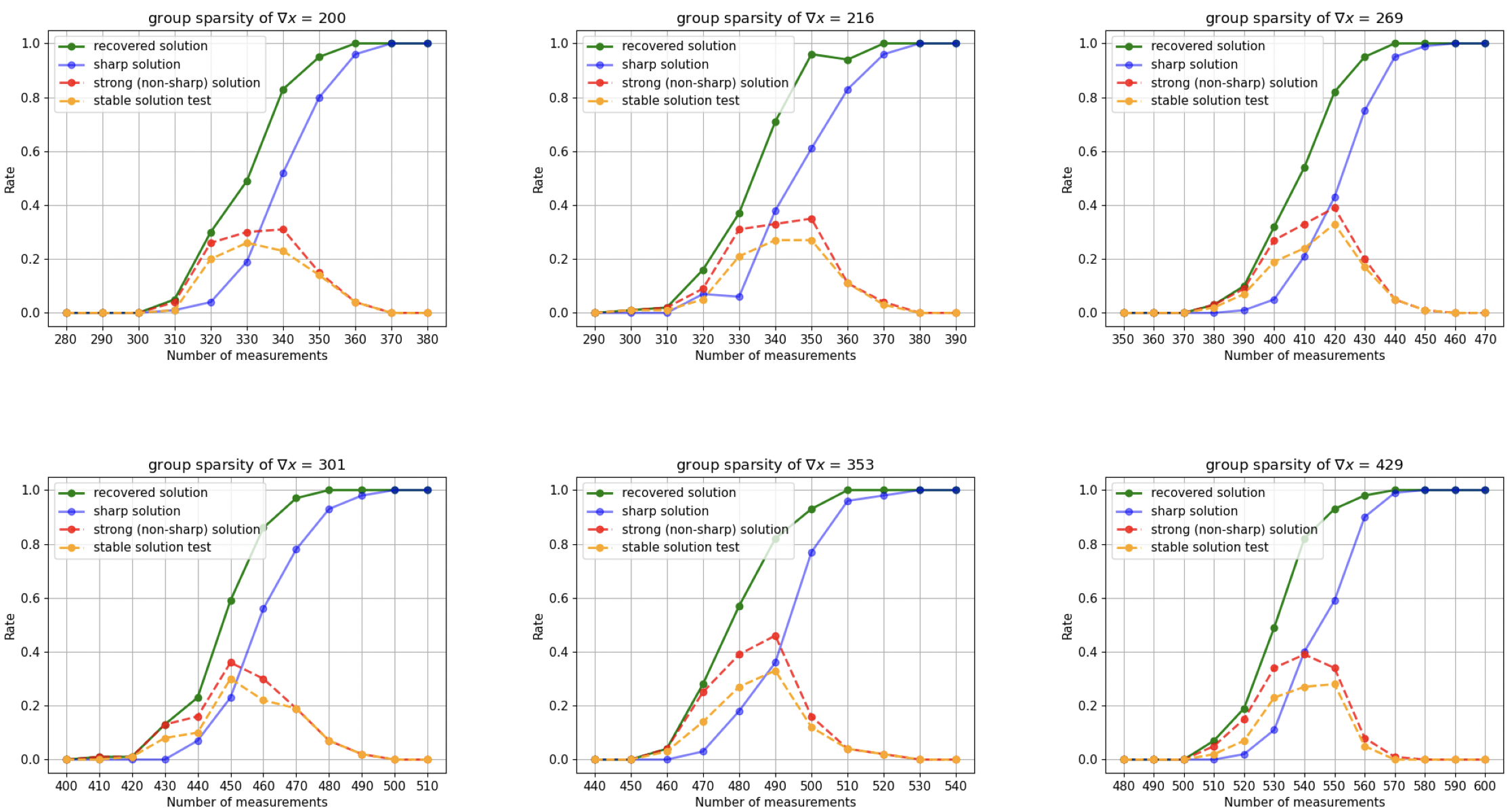} 
\caption{Isotropic total variation problems with different group sparsity of the images' gradients.}
\label{low_sparse}
\end{figure}

Each graph in Figure \ref{low_sparse} corresponds to a signal image in Figure \ref{random_img}, with group sparsity of the image gradients indicated at the top to illustrate how curve behavior varies with sparsity. The results in all graphs consistently show that there is a gap at most $40\%$ between the occurrence rates of strong (non-sharp) minima and sharp minima in scenarios with limited measurements. In every simulation, our conditions demonstrate its high capability in capturing the stable recovery of strong (non-sharp) minima. The success rate of classifying stable solutions over strong solutions via our condition~\eqref{con:Now} remains above $60\%$ in all tests and in some scenarios, exceeds $80\%$. This highlights  the effectiveness and practicality of the sufficient condition \eqref{con:Now} in image processing, particularly in the low-measurement regime, where achieving stable recovery is most challenging due to the lack of sharp minima. 


\section{Conclusion}
This paper presents comprehensive characterizations of stable recovery for low-complexity regularizers of ill-posed inverse problems, with particular attention to analysis group sparsity. Our results reveal that stable recovery inherently involves second-order structure, marking a departure from existing approaches that rely on first-order convex analysis. We demonstrate that stably recoverable solutions can be unique and strong minimizers, though not necessarily sharp. We also establish new verifiable sufficient conditions for stable recovery, which are supported by numerical experiments on group sparsity and isotropic total variation problems.

Future work involves extending our analysis to the case of nuclear norm, which is widely used in low-rank recovery problems. We also aim to deepen our understanding of the isotropic total variation problems, for which our current sufficient conditions exhibit a gap between strong and stably recoverable solutions. Bridging this gap remains an open and important problem. Furthermore, we plan to explore the role of Restricted Isometry Properties (RIP) \cite{CT05} in validating our conditions and to develop probabilistic bounds on the number of measurements required to ensure stable recovery with high probability.

\small
\bibliographystyle{abbrv}
\bibliography{bio} 

\begin{thebibliography}{10}

\bibitem{AF90}
J.-P. Aubin and H.~Frankowska.
\newblock {\em Set-valued analysis}, volume~2 of {\em Systems \& Control: Foundations \& Applications}.
\newblock Birkh\"auser Boston, Inc., Boston, MA, 1990.

\bibitem{BLN21}
Y.~Bello-Cruz, G.~Li, and T.~T. Nghia.
\newblock On the linear convergence of forward--backward splitting method: Part i—convergence analysis.
\newblock {\em Journal of Optimization Theory and Applications}, 188:378--401, 2021.

\bibitem{BB18}
M.~Benning and M.~Burger.
\newblock Modern regularization methods for inverse problems.
\newblock {\em Acta Numerica}, 27:1--111, 2018.

\bibitem{BZ87}
A.~Blake and A.~Zisserman.
\newblock {\em Visual reconstruction}.
\newblock MIT Press Series in Artificial Intelligence. MIT Press, Cambridge, MA, 1987.

\bibitem{BS00}
J.~F. Bonnans and A.~Shapiro.
\newblock {\em Perturbation analysis of optimization problems}.
\newblock Springer Science \& Business Media, 2000.

\bibitem{BDE09}
A.~M. Bruckstein, D.~L. Donoho, and M.~Elad.
\newblock From sparse solutions of systems of equations to sparse modeling of signals and images.
\newblock {\em SIAM Rev.}, 51(1):34--81, 2009.

\bibitem{CR13}
E.~Cand\`es and B.~Recht.
\newblock Simple bounds for recovering low-complexity models.
\newblock {\em Math. Program.}, 141(1-2, Ser. A):577--589, 2013.

\bibitem{CR09}
E.~J. Cand\`es and B.~Recht.
\newblock Exact matrix completion via convex optimization.
\newblock {\em Found. Comput. Math.}, 9(6):717--772, 2009.

\bibitem{CRT06}
E.~J. Candes, J.~K. Romberg, and T.~Tao.
\newblock Stable signal recovery from incomplete and inaccurate measurements.
\newblock {\em Communications on Pure and Applied Mathematics}, 59(8):1207--1223, 2006.

\bibitem{CT05}
E.~J. Candes and T.~Tao.
\newblock Decoding by linear programming.
\newblock {\em IEEE Trans. Inform. Theory}, 51(12):4203--4215, 2005.

\bibitem{CP11}
A.~Chambolle and T.~Pock.
\newblock A first-order primal-dual algorithm for convex problems with applications to imaging.
\newblock {\em J. Math. Imaging Vision}, 40(1):120--145, 2011.

\bibitem{CRPW12}
V.~Chandrasekaran, B.~Recht, P.~A. Parrilo, and A.~S. Willsky.
\newblock The convex geometry of linear inverse problems.
\newblock {\em Foundations of Computational Mathematics}, 12(6):805--849, 2012.

\bibitem{GSJA17}
G.~Cohen, S.~Afshar, J.~Tapson, and A.~van Schaik.
\newblock Emnist: Extending mnist to handwritten letters.
\newblock In {\em 2017 International Joint Conference on Neural Networks (IJCNN)}, pages 2921--2926, 2017.

\bibitem{C78}
L.~Cromme.
\newblock Strong uniqueness: A far-reaching criterion for the convergence analysis of iterative procedures.
\newblock {\em Numerische Mathematik}, 29(2):179--193, 1978.

\bibitem{DB16}
S.~Diamond and S.~Boyd.
\newblock {CVXPY}: {A} {P}ython-embedded modeling language for convex optimization.
\newblock {\em Journal of Machine Learning Research}, 17(83):1--5, 2016.

\bibitem{DET05}
D.~L. Donoho, M.~Elad, and V.~N. Temlyakov.
\newblock Stable recovery of sparse overcomplete representations in the presence of noise.
\newblock {\em IEEE Transactions on Information Theory}, 52(1):6--18, 2005.

\bibitem{DX01}
D.~L. Donoho and X.~Huo.
\newblock Uncertainty principles and ideal atomic decomposition.
\newblock {\em IEEE Trans. Inform. Theory}, 47(7):2845--2862, 2001.

\bibitem{EHN96}
H.~W. Engl, M.~Hanke, and A.~Neubauer.
\newblock {\em Regularization of Inverse Problems}, volume 375.
\newblock Springer Science \& Business Media, 1996.

\bibitem{FNP23}
J.~Fadili, T.~T. Nghia, and D.~N. Phan.
\newblock Geometric characterizations for strong minima with applications to nuclear norm minimization problems.
\newblock {\em arXiv preprint arXiv:2308.09224}, 2023.

\bibitem{FNP24}
J.~Fadili, T.~T. Nghia, and D.~N. Phan.
\newblock Solution uniqueness of convex optimization problems via the radial cone.
\newblock {\em arXiv preprint arXiv:2401.10346}, 2024.

\bibitem{FNT23}
J.~Fadili, T.~T. Nghia, and T.~T. Tran.
\newblock Sharp, strong and unique minimizers for low complexity robust recovery.
\newblock {\em Information and Inference: A Journal of the IMA}, 12(3):1461--1513, 2023.

\bibitem{FPVDS13}
J.~Fadili, G.~Peyr{\'e}, S.~Vaiter, C.~Deledalle, and J.~Salmon.
\newblock Stable recovery with analysis decomposable priors.
\newblock {\em International Conference on Sampling Theory and Application}, 2013.

\bibitem{FR13}
S.~Foucart and H.~Rauhut.
\newblock {\em A mathematical introduction to compressive sensing}.
\newblock Applied and Numerical Harmonic Analysis. Birkh\"auser/Springer, New York, 2013.

\bibitem{G11}
M.~Grasmair.
\newblock Linear convergence rates for {T}ikhonov regularization with positively homogeneous functionals.
\newblock {\em Inverse Problems}, 27(7):075014, 16, 2011.

\bibitem{GSH11}
M.~Grasmair, O.~Scherzer, and M.~Haltmeier.
\newblock Necessary and sufficient conditions for linear convergence of $\ell_1$-regularization.
\newblock {\em Communications on Pure and Applied Mathematics}, 64(2):161--182, 2011.

\bibitem{G24}
{Gurobi Optimization, LLC}.
\newblock {Gurobi Optimizer Reference Manual}, 2024.

\bibitem{H13}
M.~Haltmeier.
\newblock Block-sparse analysis regularization of ill-posed problems via $l^{2,1}$-minimization.
\newblock In {\em 2013 18th International Conference on Methods \& Models in Automation \& Robotics (MMAR)}, pages 520--523. IEEE, 2013.

\bibitem{HKS23}
J.~He, C.~Kan, and W.~Song.
\newblock On solution uniqueness and robust recovery for sparse regularization with a gauge: from dual point of view.
\newblock {\em arXiv preprint arXiv:2312.11168}, 2023.

\bibitem{H73}
P.~J. Huber.
\newblock Robust regression: asymptotics, conjectures and {M}onte {C}arlo.
\newblock {\em Ann. Statist.}, 1:799--821, 1973.

\bibitem{M93}
V.~A. Morozov.
\newblock {\em Regularization Methods for Ill-Posed Problems}.
\newblock CRC Press, Boca Raton, 1993.

\bibitem{A19}
A.~Paszke, S.~Gross, F.~Massa, A.~Lerer, J.~Bradbury, G.~Chanan, T.~Killeen, Z.~Lin, N.~Gimelshein, L.~Antiga, A.~Desmaison, A.~Kopf, E.~Yang, Z.~DeVito, M.~Raison, A.~Tejani, S.~Chilamkurthy, B.~Steiner, L.~Fang, J.~Bai, and S.~Chintala.
\newblock Pytorch: An imperative style, high-performance deep learning library.
\newblock In H.~Wallach, H.~Larochelle, A.~Beygelzimer, F.~d\textquotesingle Alch\'{e}-Buc, E.~Fox, and R.~Garnett, editors, {\em Advances in Neural Information Processing Systems}, volume~32. Curran Associates, Inc., 2019.

\bibitem{P87}
B.~T. Polyak.
\newblock {\em Introduction to Optimization}.
\newblock New York, Optimization Software, 1987.

\bibitem{RRN12}
N.~Rao, B.~Recht, and R.~Nowak.
\newblock Universal measurement bounds for structured sparse signal recovery.
\newblock In {\em Artificial Intelligence and Statistics}, pages 942--950. PMLR, 2012.

\bibitem{R70}
R.~T. Rockafellar.
\newblock {\em Convex analysis}.
\newblock Princeton, 1970.

\bibitem{RW98}
R.~T. Rockafellar and R.~J.-B. Wets.
\newblock {\em Variational Analysis}.
\newblock Springer Verlag, 1998.

\bibitem{ROF92}
L.~I. Rudin, S.~Osher, and E.~Fatemi.
\newblock Nonlinear total variation based noise removal algorithms.
\newblock {\em Phys. D}, 60(1-4):259--268, 1992.
\newblock Experimental mathematics: computational issues in nonlinear science (Los Alamos, NM, 1991).

\bibitem{TA77}
A.~N. Tikhonov and V.~J. Arsenin.
\newblock {\em Solutions of Ill-Posed Problems}.
\newblock Winston, 1977.

\bibitem{VGFP15}
S.~Vaiter, M.~Golbabaee, J.~Fadili, and G.~Peyr{\'e}.
\newblock Model selection with low complexity priors.
\newblock {\em Information and Inference: A Journal of the IMA}, 4(3):230--287, 2015.

\bibitem{VPF15}
S.~Vaiter, G.~Peyr{\'e}, and J.~Fadili.
\newblock Low complexity regularization of linear inverse problems.
\newblock {\em Sampling Theory, a Renaissance: Compressive Sensing and Other Developments}, pages 103--153, 2015.

\bibitem{YL06}
M.~Yuan and Y.~Lin.
\newblock Model selection and estimation in regression with grouped variables.
\newblock {\em Journal of the Royal Statistical Society Series B: Statistical Methodology}, 68(1):49--67, 2006.

\bibitem{ZH05}
H.~Zou and T.~Hastie.
\newblock Regularization and variable selection via the elastic net.
\newblock {\em J. R. Stat. Soc. Ser. B Stat. Methodol.}, 67(2):301--320, 2005.

\end{thebibliography}
\end{document}